\documentclass[final, onefignum,onetabnum]{arxiv}

\usepackage{lipsum}
\usepackage{amsfonts}
\usepackage{graphicx}
\usepackage{epstopdf}
\usepackage{algorithmic}

\usepackage{amsmath,amssymb}
\usepackage{booktabs}
\usepackage{enumerate}

\newcommand{\less}{\leqslant}
\newcommand{\gre}{\geqslant}
\newcommand{\what}{\widehat}
\newcommand{\wtilde}{\widetilde}
\newcommand{\argmin}{\mathop{\rm argmin}}

\newcommand{\moment}{\theta}
\newcommand{\defn}{\ensuremath{: \, =}}
\newcommand{\real}{\ensuremath{\mathbb{R}}}

\newsiamremark{remark}{Remark}
\newsiamremark{hypothesis}{Hypothesis}
\newsiamremark{assumption}{Assumption}
\newsiamremark{condition}{Condition}
\crefname{hypothesis}{Hypothesis}{Hypotheses}
\newsiamthm{claim}{Claim}

\headers{Faster Least Squares Optimization}{J. Lacotte, and M. Pilanci}

\title{Faster Least Squares Optimization}

\author{Jonathan Lacotte\thanks{Department of Electrical Engineering, Stanford University, 
		(\email{lacotte@stanford.edu}).}
	\and Mert Pilanci\thanks{Department of Electrical Engineering, Stanford University
		(\email{pilanci@stanford.edu}).}}

\usepackage{amsopn}

\ifpdf
\hypersetup{
  pdftitle={Faster Least Squares Optimization},
  pdfauthor={J. Lacotte, and M. Pilanci}
}
\fi

\begin{document}

\maketitle

\begin{abstract}
We investigate iterative methods with randomized preconditioners for solving overdetermined least-squares problems, where the preconditioners are based on a data-oblivious random embedding (a.k.a. \emph{sketch}) of the data matrix. We consider two distinct approaches: the sketch is either computed once (fixed preconditioner), or, the random projection is refreshed at each iteration, i.e., sampled independently of previous ones (varying preconditioners). Although fixed sketching-based preconditioners have received considerable attention in the recent literature, little is known about the performance of refreshed sketches. For a fixed sketch, we characterize the optimal iterative method, that is, the preconditioned conjugate gradient (PCG) as well as its rate of convergence in terms of the subspace embedding properties of the random embedding. For refreshed sketches, we provide a closed-form formula for the expected error of the iterative Hessian sketch (IHS), a.k.a.~preconditioned steepest descent. In contrast to the guarantees and analysis for fixed preconditioners based on subspace embedding properties, our formula is \emph{exact} and it involves the expected inverse moments of the random projection. Our main technical contribution is to show that this convergence rate is, surprisingly, unimprovable with heavy-ball momentum. Additionally, we construct the locally optimal first-order method whose convergence rate is bounded by that of the IHS, and we relate this method to existing conjugate gradient methods with varying preconditioners (e.g., the flexible conjugate gradient and the inexact preconditioned conjugate gradient). Based on these theoretical and numerical investigations, we do not observe that the additional randomness of refreshed sketches provides a clear advantage over a fixed preconditioner. Therefore, we prescribe PCG as the method of choice along with an optimized sketch size according to our analysis. Our prescribed sketch size yields state-of-the-art computational complexity (as measured by the flop counts in an idealized RAM model) for solving highly overdetermined linear systems. Lastly, we illustrate the numerical benefits of our algorithms.
\end{abstract}


\section{Introduction}

We consider the convex quadratic program
\begin{align}
\label{eqnconvexquadratic}
    x^* \defn \argmin_{x \in \real^d} \left\{f(x) \defn  \frac{1}{2} \langle x, Hx \rangle - \langle b, x\rangle\right\}\,.
\end{align}
where $H = A^\top A$ for a given data matrix $A \in \real^{n \times d}$ with $n \gre d$ and $n \gg 1$, and $b \in \real^d$. In this work, we are interested in the following class of pre-conditioned first-order methods
\begin{align}
\label{eqnfirstordermethods}
    x_{t+1} \in x_0 + \mbox{span}\!\left\{H_{S_0}^\dagger \nabla f(x_0), \dots, H_{S_{t}}^\dagger \nabla f(x_t)\right\}\,,
\end{align}
where $x_0 \in \real^d$ is a given initial point, the matrices $\{S_t\}_{t \gre 0}$ are $m \times n$, and the approximate Hessian $H_{S_t}$ at time $t \gre 0$ is defined as 
\begin{align}
    H_{S_t} \defn A^\top S_t^\top S_t A\,.
\end{align}
We consider \emph{random embeddings} (or \emph{sketching matrices}) $S_t$. Classical random embeddings include Gaussian embeddings with independent entries $\mathcal{N}(0,1/m)$ and Haar matrices as well as the subsampled randomized Hadamard transform (SRHT) and the sparse Johnson Lindenstrauss transform (SJLT) -- see Section~\ref{sectionsketchingmatrices} for background.

Randomized pre-conditioned iterative methods are now standard solvers for modern convex quadratic programs. In the large-scale setting, direct factorization methods have prohibitive computational cost $\mathcal{O}(nd^2)$ whereas the performance of standard iterative solvers such as the conjugate gradient method (CG) have condition number dependency. On the other hand, the current literature on randomized pre-conditioned iterative solvers lacks a unifying perspective, and we aim to take a step towards this goal. We are interested in two settings: first, the sketching matrices are all equal, i.e., $S_t = S$ for a \emph{fixed} sketching matrix $S \in \real^{m \times n}$; second, the sketching matrices $\{S_t\}$ are independent and identically distributed (i.i.d.) and we say that they are \emph{refreshed}.

\subsection{Iterative Hessian sketch} For fixed or refreshed embeddings, the canonical instances of~\eqref{eqnfirstordermethods} are the iterative Hessian sketch (IHS) and the IHS~with heavy-ball momentum (Polyak-IHS), whose updates are respectively given by
\begin{align}
    & x_{t+1} = x_t - \mu \, H_{S_t}^\dagger \nabla f(x_t) & \mbox{(IHS)} \label{eqnihs}\\
    & x_{t+1} = x_t - \mu\, H_{S_t}^\dagger \nabla f(x_t) + \beta \, (x_t - x_{t-1}) &\mbox{(Polyak-IHS)} \label{eqnpolyakihs}
\end{align}
These two methods can respectively be viewed as gradient descent and the Chebyshev semi-iterative method~\cite{manteuffel1977tchebychev} with pre-conditioners $H_{S_t}$.

\subsection{Preconditioned CG with fixed or variable preconditioners} The class~\eqref{eqnfirstordermethods} also contains the preconditioned conjugate gradient method with a fixed embedding $S$ (PCG); see, e.g., Ch.~5 in~\cite{nocedal2006numerical} for more background. Specialized to a fixed preconditioner $H_{S}$, PCG starts with $r_0 = b - H x_0$, solves $H_S \what r_0 = r_0$, sets $p_0 = \what r_0$, and then computes at each iteration 
\begin{align}
    x_{t+1} = x_t + \alpha_t p_t\,,\,\,\, r_{t+1} = r_t - \alpha_t H p_t\,,\,\,\, p_{t+1} = \what r_{t+1} + \frac{r_{t+1}^\top \what r_{t+1}}{r_t^\top \what r_t}\cdot  p_t \quad \mbox{(PCG)}
\end{align}
where $\alpha_t = \frac{r_t^\top \what r_t}{p_t^\top H p_t}$, and $\what r_{t+1}$ is solution of the linear system $H_S \cdot \what r_{t+1} = r_{t+1}$. Efficient implementations of PCG are based on precomputing a factorization $H_S = (H_S^\frac{1}{2})^\top H_S^\frac{1}{2}$ such that the latter linear system can be efficiently solved, e.g., $H_S^\frac{1}{2}$ is a product of orthogonal, diagonal and/or triangular matrices. For instance,~\cite{rokhlin2008fast, avron2010blendenpik} propose to compute a pivoted QR decomposition $SA = QR \Pi$, which corresponds to $H_S^\frac{1}{2} = R \Pi$. It is proposed in~\cite{meng2014lsrn} to compute a thin SVD $SA = \wtilde U \wtilde \Sigma \wtilde V^\top$, which corresponds to $H_S^\frac{1}{2} = \wtilde \Sigma^\frac{1}{2} \wtilde V^\top$. 

A canonical extension of PCG to the case of variable pre-conditioners is the so-called generalized conjugate gradient method (GCC); see for instance~\cite{axelsson1991black, knyazev2008steepest}. Specialized to the case of refreshed embeddings and the preconditioners $H_{S_t}$, it computes at each iteration $t \gre 0$ the direction $v_t = H_{S_t}^\dagger \nabla f(x_t)$ and
\begin{align}
\label{eqnconjugatedirectionsrefreshed}
    p_t = v_t - \sum_{j=0}^{t-1} \frac{\langle v_t, H p_j\rangle}{\langle p_j, H p_j \rangle} p_j\,,\qquad x_{t+1} = x_t + \frac{\langle p_t, b \rangle}{\langle p_t, H p_t \rangle} p_t\,,\quad \mbox{(GCC)}
\end{align}
In contrast to PCG, the update~\eqref{eqnconjugatedirectionsrefreshed} involves a full orthogonalization of the new direction $H_{S_t}^\dagger g_t$ with respect to previous search directions $p_0, \dots, p_{t-1}$. The above update is known (e.g., Lemma~2.2 in~\cite{axelsson1991black} or Lemma~3.3 in~\cite{knyazev2008steepest}) to be locally optimal within the span of the previous search directions, i.e.,
\begin{align}
\label{eqnoptimalityconditionrefreshed}
    x_{t+1} \defn \argmin_{x \in x_0 + \mathcal{K}_t} \|A(x - x^*)\|_2^2\,,
\end{align}
where $\mathcal{K}_t \defn \mbox{span}\{H_{S_0}^\dagger \nabla f(x_0), \dots, H_{S_t}^\dagger \nabla f(x_t)\}$. A natural approximation -- known as the flexible conjugate gradient method (FCG) -- is to truncate the orthogonalization step to a limited number of past conjugate directions. That is, at each iteration $t \in \{0, \dots, d-1\}$, we compute $v_t = H_{S_t}^\dagger \nabla f(x_t)$ and we do
\begin{align}
\label{eqnflexiblecg}
    p_t = v_t - \sum_{j=t-k_t}^{t-1} \frac{\langle v_t, H p_j\rangle}{\langle p_j, H p_j \rangle} p_j\,,\qquad x_{t+1} = x_t + \frac{\langle p_t, b \rangle}{\langle p_t, H p_t \rangle} p_t\,,\quad \mbox{(FCG)}
\end{align}
where $0 \less k_t \less t$ and $k_{t+1} \less k_{t}+1$. FCG has been considered in~\cite{notay2000flexible, golub1999inexact, knyazev2008steepest} in the case of deterministic variable preconditioners which aim to approximate a fixed linear operator and thus analyzed from a worst-case approximation error perspective. To our knowledge, we are the first to consider FCG with the randomized i.i.d.~preconditioners $H_{S_t}$; see Section~\ref{sectionihsrefreshed}.

In this work, we aim to provide a more unified view over the class of sketching-based preconditioned first-order methods~\eqref{eqnfirstordermethods} and we address the following questions. What is the optimal method and its convergence rate with a fixed embedding? With refreshed embeddings? Does the additional randomness of refreshed embeddings provide any advantage over a fixed preconditioner? What sketching-based method shall be prescribed for solving a generic, dense linear system and what sketch size?

\subsection{Statement of Contributions}

We will show that the performance of the class of pre-conditioned methods~\eqref{eqnfirstordermethods} solely depends on the spectral properties of the matrices $C_{S_t} \defn U^\top S_t^\top S_t U$, where $U \in \real^{n \times d}$ is the matrix of left singular vectors of $A$, and this remarkable property holds universally for any data pair $(A,b)$. Importantly, for large enough $m$ (e.g., $m \gtrsim d$ for Gaussian embeddings or $m \gtrsim d \log d$ for the SRHT), the condition number of $C_{S_t}$ is close to $1$. In order to limit the scope of our discussion, we focus most of our analysis around Gaussian and Haar whose spectral properties are finely characterized, as well as the SRHT. Nonetheless, many of our results readily extend to a larger class of embeddings. More specifically, we have the following contributions.

\textbf{Fixed sketch and extreme eigenvalues analysis.} Given a fixed embedding $S \in \real^{m \times n}$, we show that PCG is the optimal method within the class~\eqref{eqnfirstordermethods}. Our proof of this fact is based on standard optimality results of CG that we adapt to our setting, i.e., we relate its convergence rate to the eigenvalues of $C_S$ and we show that the required number of iterations to attain a certain precision is fully predictable in terms of the extreme eigenvalues of $C_S$. We recall similar convergence results (already established in~\cite{ozaslan2019iterative}) for the IHS and the Polyak-IHS, as these methods are known to perform better in certain settings (e.g., on clusters with high-communication costs). 

\textbf{Refreshed sketches and moments-based analysis.} Given i.i.d.~embeddings $S_t \in \real^{m \times n}$, we provide an exact error formula \emph{in expectation} for the IHS and characterize the exact optimal step size. Our analysis is novel and it involves the inverse moments $\mathbb{E}\{C_S^{-j}\}$ for $j \in \{1,2\}$ in contrast to standard results in the literature which are based on extreme eigenvalues analysis and condition numbers. We construct the locally optimal method (GCC) and its flexible version (FCG), and we show that these methods inherit the error bound of the IHS.

\textbf{Heavy-ball momentum and refreshed sketches.} Given i.i.d.~embeddings $S_t \in \real^{m \times n}$, we prove that Polyak-IHS with a constant momentum parameter does not provide acceleration, and this is our main technical contribution. This remarkable fact contrasts with standard results on the performance of first-order methods (e.g., Chebyshev semi-iterative method). Furthermore, we observe numerically that GCC with refreshed embeddings does not outperform PCG with a fixed embedding, despite the additional randomness. Based on these numerical and theoretical comparisons, we prescribe PCG with a fixed preconditioner as a generic method of choice. 

\textbf{Optimized sketch size and faster least-squares optimization.} We characterize the sketch size that minimizes the computational cost (as measured by the flops count in an idealized RAM model) of PCG for both Gaussian embeddings and the SRHT in order to reach an $\varepsilon$-accurate solution. In the case of highly overparameterized problems ($n > d^2$), high precision and the SRHT, we obtain the optimized complexity
\begin{align}
    \mathcal{C} = \mathcal{O}\!\left(nd (\log d + \frac{\log(1/\varepsilon)}{\log(n/d^2)})\right)\,.
\end{align}
This improves on the classical cost $\mathcal{C}_\text{cl} = \mathcal{O}(nd (\log d + \log(1/\varepsilon)))$ with the standard prescription (see, e.g.,~\cite{rokhlin2008fast}) $m=\mathcal{O}(d\log d)$. To our knowledge, our sketch size prescription yields the state-of-the-art complexity for solving dense linear systems, and we illustrate numerically its benefits.

\subsection{Open Questions}

We do not characterize the \emph{globally} optimal method with refreshed embeddings, nor its rate of convergence, and we leave this as an open problem. Furthermore, we establish upper bounds on the respective errors of PCG and GCC, and we show that for our embeddings of interest, these upper bounds scale similarly. We then choose PCG as the method of choice based on numerical comparisons. However, it is left open to determine whether PCG outperforms GCC uniformly over the dimensions $n, d, m$ and the choice of the embedding. More generally, we leverage well-known facts about the spectral properties of the embeddings we consider. However, we do not provide a systematic way of comparing the results of our extreme eigenvalues analysis for a fixed embedding and of our moments-based analysis for refreshed embeddings. It is thus left open to characterize whether the bounds we provide for, e.g., GCC versus PCG can be compared for other interesting classes of embeddings for which we do not know, for instance, the inverse moments $\mathbb{E}\{C_S^{-j}\}$ with $j \in \{1,2\}$ or their trace, e.g., the SJLT.

\subsection{Notations}
We reserve the notations $0 \less \lambda_d \less \hdots \less \lambda_1$ for the eigenvalues of the matrix $C_S$, and $C_S = \sum_{i=1}^d \lambda_i v_i v_i^\top$ for an eigenvalue decomposition. We denote by $\|\cdot\|_H$ the norm induced by $H$, i.e., $\|x\|_H^2 = \langle x, H x\rangle$. We denote a thin SVD $A = U \Sigma V^\top$ and we define the shorthand $H^\frac{1}{2} = U^\top A$. Note that $H = (H^\frac{1}{2})^\top H^\frac{1}{2}$. Given an iterate $x_t$, we introduce the error vector $\Delta_t = H^\frac{1}{2} (x_t - x^*)$ and the prediction error $\delta_t = \frac{1}{2} \|x_t - x^*\|_H^2$. Note that $\delta_t = \frac{1}{2} \|\Delta_t\|_2^2$. We will assume for simplicity that the matrix $A$ is full-column rank and that the random projections $S_t A$ are also full-column rank almost surely. For instance, with Gaussian embeddings, this holds for $m \gre d$.

\section{Randomized Sketches}
\label{sectionsketchingmatrices}

Gaussian embeddings, i.e., matrices $S \in \real^{m \times n}$ with i.i.d.~Gaussian entries $\mathcal{N}(0,1/m)$ are a classical random projection whose spectral and subspace embedding properties are tightly characterized. The cost of forming the sketch $S \cdot A$ for a dense matrix $A$ requires $\mathcal{O}(ndm)$ flops (using classical matrix multiplication). In practice, e.g., parallelized computation or for a sparse matrix $A$, the running time can actually be significantly faster (see, e.g.,~\cite{meng2014lsrn} for a detailed discussion of practical advantages of Gaussian embeddings). Haar embeddings also have strong subspace embedding properties, but are slow to generate. A matrix $S \in \real^{m \times n}$ is a Haar embedding if $S^\top S = I_m$ and if its span is uniformly distributed among the $m$-dimensional subspaces of $\real^n$. Another orthogonal transform is the SRHT~\cite{ailon2006approximate}, and it provides in general a more favorable trade-off in terms of subspace embedding properties and sketching time. A matrix $S \in \real^{m \times n}$ is a SRHT if $S = R H E$ where $R \in \real^{m \times n}$ is a row-subsampling matrix (uniformly at random without replacement), $H \in \real^{n \times n}$ is the Hadamard matrix\footnote{The Hadamard transform is defined for $n = 2^k$ for some $k \gre 0$. If $n$ is not a power of $2$, a standard practice is to pad the data matrix with $2^{\lceil \log_2(n)\rceil}-n$ additional rows of $0$'s.} of size $n$, and $E$ is a diagonal matrix with random signs on the diagonal. Its sketching cost is near-linear and scales as $\mathcal{O}(nd\log m)$. In this work, we will also consider the general class of randomized embeddings that satisfy the following condition.
\begin{condition}
\label{condition:unbiasedmoments}
A random embedding $S \in \real^{m \times n}$ satisfies the first and second inverse unbiased moments condition if there exists $m_0 \gre 1$ such that for any $m \gre m_0$, there exist finite real numbers $\theta_1, \theta_2 > 0$ such that
\begin{align}
\label{eqninversemoments}
	\mathbb{E}\{C_S^{-1}\} = \theta_1 \cdot I_d\,,\quad \mbox{and} \quad \mathbb{E}\{C_S^{-2}\} = \theta_2 \cdot I_d\,,
\end{align}
for any matrix $U \in \real^{n \times d}$ with orthonormal columns. 
\end{condition}
Note that we have the trace formula $\theta_j = d^{-1}\mbox{tr}\,\mathbb{E}\{C_S^{-j}\}$ for $j \in \{1,2\}$. Due to their rotational invariance, Gaussian and Haar embeddings satisfy Condition~\ref{condition:unbiasedmoments}. According to Lemma~2.3 in~\cite{gupta1968some}, it holds for Gaussian embeddings that
\begin{align}
\label{eqngaussianmoments}
    \theta_1 =  \frac{m}{m-d-1} \,,\quad \mbox{and} \quad \theta_2 = \frac{m^2 (m-1)}{(m-d)(m-d-1)(m-d-3)}\,,
\end{align}
provided that $m \gre d+4$. According to~\cite{lacotte2020limiting} (see Lemma~3.2 therein), we have for Haar matrices the finite-sample approximations
\begin{align}
\label{eqnsrhtmoments}
    \theta_1 \approx  \frac{n-d}{m-d} \,,\quad \mbox{and} \quad \theta_2 \approx \frac{n-d}{(m-d)^3}(d^2 + mn - 2dm)\,.
\end{align}
It is not known whether the SRHT does satisfy even approximately Condition~\ref{condition:unbiasedmoments} (see Section 5.3 in~\cite{derezinski2020sparse} for more insights). Nonetheless, it has been shown (see Lemma~4.3 in~\cite{lacotte2020limiting}) that the trace formula $d^{-1} \mbox{tr}\, \mathbb{E}\{C_S^{-j}\}$ for $j \in \{1,2\}$ are asymptotically the same for the SRHT and Haar matrices. That is, it holds asymptotically for the SRHT that
\begin{align}
\label{eqnsrhttrace}
    \frac{1}{d}\mbox{tr}\, \mathbb{E}\{C_S^{-1}\} \approx  \frac{n-d}{m-d} \,,\quad \mbox{and} \quad \frac{1}{d}\mbox{tr}\, \mathbb{E}\{C_S^{-2}\} \approx \frac{n-d}{(m-d)^3}(d^2 + mn - 2dm)\,.
\end{align}
Consequently, although some of our formal guarantees (see Section~\ref{sectionihsrefreshed}) will be stated for Haar embeddings, we will carry out some of our numerical experiments with the SRHT instead. In fact, after the appearance of a preliminary version of this work, the follow-up work~\cite{lacotte2020limiting} showed that the IHS with refreshed SRHT and Haar embeddings have similar numerical performance.

Another common choice in practice is a sparse embedding such as the SJLT~\cite{clarkson2017low} (see, also, the more general class of OSNAPs~\cite{nelson2013osnap}): for each column, a number $s$ of rows are chosen uniformly at random without replacement, and the corresponding entries are randomly chosen in $\{\pm 1/\sqrt{s}\}$. However, little is known regarding the trace formula of the SJLT and whether it satisfies even approximately Condition~\ref{condition:unbiasedmoments} (see again the discussion in Section 5.3 in~\cite{derezinski2020sparse}). Consequently, we choose to not address sparse embeddings in this work. Lastly, the authors of~\cite{derezinski2020sparse} recently introduced the so-called LESS embeddings whose inverse moments $\mathbb{E}\{C_S^{-1}\}$ and $\mathbb{E}\{C_S^{-2}\}$ are nearly unbiased and whose sketching cost scales as $\mathcal{O}(nd \log m)$. As for future work, it may be of great practical interest to extend the present analysis to such embeddings with small inversion biases.

\section{Pre-conditioned First-order Methods with Fixed Embedding}
\label{sectionfixed}

We analyze first-order methods with a fixed sketching-based preconditioner. Our error guarantees are exact and involve the eigenvalues of the matrix $C_S$. Standard upper bounds are derived from its extreme eigenvalues $\lambda_1$ and $\lambda_d$. Provided these extreme eigenvalues can be estimated (e.g., with high-probabililty bounds), the algorithms we develop in this section are then implementable with a fully predictable number of iterations to converge up to a certain precision. In particular, given a fixed embedding $S \in \real^{m \times n}$ and a deviation parameter $\rho \in (0,1)$, we consider the $S$-measurable event 
\begin{align}
    \mathcal{E}_{\rho}^m \defn \left\{ (1-\sqrt{\rho})^2 \less \lambda_d \less \lambda_1 \less (1+\sqrt{\rho})^2 \right\}\,.
\end{align}
We show that PCG is the globally optimal method. We also develop error guarantees for the IHS and Polyak-IHS which are weaker than PCG. These algorithms may be of greater practical interest in settings with high communication costs due to the extra synchronization steps of PCG~\cite{meng2014lsrn, varga1961}.

\subsection{Optimal Method and Error Guarantees}

We establish first a lower bound on the performance of any first-order method and that it is attained by PCG. To prove the latter, we leverage the well-known fact (e.g.,~\cite{d2018optimization}, Ch.~1, or~\cite{axelsson1996iterative}, Section~11.1.2) that PCG is equivalent to standard CG applied to the preconditioned quadratic function $f_S : z \mapsto \frac{1}{2} \|A H_S^{-\frac{1}{2}} z\|_2^2 - b^\top H_S^{-\frac{1}{2}} z$ and starting at $z_0 = H_S^{\frac{1}{2}} x_0$. Precisely, CG applied to $f_S$ returns a sequence of iterates $\{z_t\}$ such that $x_t = H_S^{-\frac{1}{2}} z_t$ for all $t \gre 0$. Hence, we can leverage standard optimality results of CG to establish the next result.
\begin{theorem}[Global Optimality of PCG]
\label{theoremoptimalitypcg}
    Let $S \in \real^{m \times n}$, and denote $\xi_i = \langle v_i, \Delta_0 \rangle$. It holds for any pre-conditioned first-order method with the fixed embedding $S$ and starting at $x_0$ that
    \begin{align}
    \label{eqnlowerboundfixedsketch}
        \delta_t \gre \Big\{\ell_{t}^*(S,x_0) \defn \min_{\substack{Q \in \real_t[X];\\Q(0)=1}} \sum_{i=1}^d Q(\lambda^{-1}_i)^2 \xi_i^2 \Big\}\,.    
    \end{align}
    Furthermore, PCG~starting at $x_0$ attains the lower bound $\ell^*_t(S,x_0)$.
\end{theorem}
The exact error $\ell^*_t(S,x_0)$ of PCG depends on all eigenvalues of $C_S$; it seems arduous to characterize it more explicitly (e.g., in a non-variational form and in terms of the dimensions $n,d,m$) and it is expensive to compute it numerically. We will use instead the following classical upper bound (e.g., Theorem~1.2.2 in~\cite{daniel1967conjugate}) in terms of the extreme eigenvalues $\lambda_1$ and $\lambda_d$. That is, we have $\ell_t^*(S,x_0) \less 4 \cdot \left(\frac{\sqrt{\lambda_1}-\sqrt{\lambda_d}}{\sqrt{\lambda_1}+\sqrt{\lambda_d}}\right)^{2t} \cdot \delta_0$. Therefore, conditional on $\mathcal{E}^m_{\rho}$, we obtain for PCG that
\begin{align}
\label{eqnupperboundpcgrate}
    \frac{\delta_t}{\delta_0} \less 4 \cdot \rho^t\,.
\end{align}
Combining~\eqref{eqnupperboundpcgrate} and high-probability bounds on the extreme eigenvalues of the matrix $C_S$ (see Table~\ref{tableembeddingproperties}), we have the following error bounds for PCG. For Gaussian embeddings, it holds with probability $1-e^{-\mathcal{O}(d)}$ that 
\begin{align}
\label{eqnerrorpcggaussian}
    \frac{\delta_t}{\delta_0} \less 4 \cdot \left(\frac{d}{m}\right)^t\,,
\end{align}
and for the SRHT, it holds with probability $1-\mathcal{O}(1/d)$ that
\begin{align}
\label{eqnerrorpcgsrht}
    \frac{\delta_t}{\delta_0} \less 4 \cdot \left(\frac{d \log d}{m}\right)^t\,.
\end{align}
\begin{table}[!h]
\caption{Given $\delta \in (0,1/2)$, we recall the critical sketch size $m_{\delta} \defn \inf \{k \gre 1 \mid \mathbb{P}(\mathcal{E}^m_\rho) \gre 1-\delta\,,\forall \rho \in (0,1), \,\forall m \gre k / \rho\}$ for some classical random embeddings: Gaussian and sub-Gaussian embeddings, the SRHT and the SJLT with $s$ non-zero entries per column sampled uniformly at random without replacement; we refer to~\cite{vershynin2018high, cohen2015optimal, nelson2013osnap} for detailed proofs of these results.}
\label{tableembeddingproperties}
    \centering
    \begin{tabular}{|c|c|}
    \cmidrule(r){1-2}
         Embedding & Critical sketch size $m_{\delta}$  \\
        \midrule 
         Gaussian & $d + \mathcal{O}(\log(1/\delta))$\\
         sub-Gaussian & $\mathcal{O}\!\left(d + \log(1/\delta))\right)$ \\
         SRHT & $\mathcal{O}\!\left(\log(d /\delta) \cdot \left(d +  \log(n/\delta)\right) \right)$\\
         SJLT, $s=1$& $\mathcal{O}\big(d^2 / \delta\big)$ \\
    \bottomrule
    \end{tabular}
\end{table}
In comparison to PCG, we have the following weaker error guarantees for the IHS and Polyak-IHS. We emphasize that a similar result has been established in the concurrent work~\cite{ozaslan2019iterative}.
\begin{theorem}
\label{theoremihsfixed} 
For a fixed embedding $S \in \real^{m \times n}$, conditional on $\mathcal{E}^m_{\rho}$, it holds that the IHS with constant step size $\mu = \frac{(1-\rho)^2}{1+\rho}$ satisfies at every iteration 
\begin{align} 
\label{eqnerrorihs}
    \frac{\delta_t}{\delta_0} \less \left( \frac{4 \rho}{(1+\rho)^2} \right)^{t}\,. 
\end{align}
Furthermore, conditional on $\mathcal{E}_\rho^m$, the Polyak-IHS with constant parameters $\mu = (1-\rho)^2$ and $\beta = \rho$ satisfies
\begin{align} 
\label{eqnerrorpolyakihs} 
    \limsup_{t \to \infty} \left(\frac{\delta_t}{\delta_0}\right)^\frac{1}{t} \less \rho\,.
\end{align}
\end{theorem}
The convergence guarantee for the Polyak IHS is asymptotic, and this is essentially due to the approximation error of the spectral radius by Gelfand's formula (see, e.g.,~\cite{kozyakin2009}). However, as for PCG, the \emph{expected} error of the IHS can be characterized exactly in terms of all the eigenvalues of the matrix $C_S$. This is a remarkable universality result which does hold independently of the data pair $(A,b)$.
\begin{theorem}
\label{thmexacterrorfixed}
Let $m \gre d$ and consider $S \in \real^{m \times n}$ a Gaussian or Haar embedding. Then, it holds that the IHS has exact expected error
\begin{align}
\label{eqnexacterrorihs}
    \mathbb{E}\{\delta_t\} = \mathrm{\Gamma}_t(\mu) \cdot \mathbb{E}\{\delta_0\}\,,
\end{align}
where $\mathrm{\Gamma}_t(\mu) = \mathbb{E}\{\frac{1}{d} \sum_{i=1}^d \big(1- \mu \lambda_i^{-1} \big)^{2t}\}$. 
\end{theorem}
In order to illustrate the benefits of this universality result, we show next that the error formula~\eqref{eqnexacterrorihs} can be more explicitly characterized in the asymptotic regime where we let the relevant dimensions $n, d, m \to +\infty$. For conciseness, we specialize our next discussion to the case of Gaussian embeddings.

\subsection{Asymptotically Exact Error Formula for the IHS with a Fixed Gaussian Embedding}

Our asymptotic results involve the Marcenko-Pastur distribution~\cite{marvcenko1967distribution} with parameter $(\rho, \sigma) \in (0,1) \times (0,+\infty)$ that we denote by $\text{MP}(\rho, \sigma)$. We recall that its density with respect to the Lebesgue measure on $\real$ is given by 
\begin{align}
    \nu(\lambda) = (2 \pi \sigma^2 \rho \lambda)^{-1} \sqrt{(\lambda_+ - \lambda)(\lambda - \lambda_-)} \mathbf{1}_{[\lambda_-,\lambda_+]}(\lambda)\,,
\end{align}
where $\lambda_- = \sigma^2 (1-\sqrt{\rho})^2$ and $\lambda_+ = \sigma^2 (1+\sqrt{\rho})^2$. We consider the asymptotic regime where $n, m, d \to \infty$, such that $d/m \to \rho \in (0,1)$. According to a classical result~\cite{marvcenko1967distribution}, the empirical distribution of the eigenvalues $\lambda_1, \hdots, \lambda_d$ of $C_S$ converges weakly to the distribution $\text{MP}(\rho,1)$. Therefore, the function $\Gamma_t(\mu)$ converges pointwise to $\mu \mapsto \Gamma^\rho_t(\mu) \defn \mathbb{E}_{\lambda \sim \text{MP}(\rho,1)}\big\{\big(1-\mu \lambda^{-1}\big)^{2t}\big\}$. The function $\Gamma^\rho_t$ is strongly convex so that it admits a unique minimizer $\mu^*_t$. For $t \to \infty$, we provide next a closed-form expression for the optimal step size and the resulting rate of convergence. That is, we are interested in finding a step size $\mu^*$, if any, which satisfies $\liminf_{t \to +\infty} \frac{\Gamma^\rho_t(\mu)}{\Gamma^\rho_t(\mu^*)} \gre 1$ for all $\mu \in \real$ and we aim to characterize the asymptotic rate of convergence $r_{\infty} = \lim_{t \to + \infty} \frac{\Gamma^\rho_{t+1}(\mu^*)}{\Gamma^\rho_t(\mu^*)}$.
\begin{theorem}
\label{theoremmp}
Let $\mu^* = \frac{\left(1-\rho\right)^2}{1+\rho}$. Then, for any $\mu \neq \mu^*$, it holds that
\begin{align*}
\lim_{t \to +\infty} \frac{\Gamma^\rho_t(\mu)}{\Gamma^\rho_t(\mu^*)} = +\infty\,,\qquad \text{and} \qquad r_\infty = \frac{4 \rho}{(1+\rho)^2}\,.
\end{align*}
\end{theorem}
Remarkably, the asymptotic rate $r_\infty$ scales as the error upper bound~\eqref{eqnerrorihs}. Hence, for a fixed Gaussian embedding, an extreme eigenvalues analysis is asymptotically exact.

\subsection{Proof of Results in Section~\ref{sectionfixed}}

\subsubsection*{Proof of Theorem~\ref{theoremoptimalitypcg}}

Fix $S \in \real^{m \times n}$ and $x_0 \in \real^d$, and consider a pre-conditioned first-order method based on $S$ and starting at $x_0$, i.e., 
\begin{align}
    x_t \in x_0 + H_S^{-1} \text{span}\!\left\{\nabla f(x_0), \dots, \nabla f(x_{t-1})\right\}    
\end{align}
at every iteration $t \gre 1$. We show by induction that, for any $t \gre 1$, there exists $Q_t \in \real_{t}[X]$ such that $Q_t(0)=1$ and $x_t-x^* = Q_t(H_S^{-1}H) (x_0 - x^*)$. We have $x_1-x^* = x_0-x^* + \alpha H_S^{-1}\nabla f(x_0)$ for some $\alpha \in \real$, i.e., $x_1 = x_0 + \alpha H_S^{-1} H (x_0 - x^*)$. Setting $Q_0(X) = 1+\alpha X$ yields the induction claim for $t=1$. Suppose that the induction hypothesis holds for some $t \gre 1$. We have $x_{t+1}-x^* = x_0-x^* + H_S^{-1} \sum_{j=0}^t \alpha_j \nabla f(x_j)$ for some $\alpha_0, \dots, \alpha_t \in \real$. Note that $\nabla f(x_j) = H (x_j - x^*)$, so that $x_{t+1}-x^* = x_0-x^* + H_S^{-1}H \sum_{j=0}^t \alpha_j (x_j - x^*)$. By induction hypothesis, for each $0 \less j \less t$, there exists $Q_j \in \real_{j}[X]$ such that $Q_j(0)=1$ and $x_j - x^* = Q_j(H_S^{-1}H) (x_0 - x^*)$. Hence, $x_{t+1}-x^* = Q_{t+1}(x_0 - x^*)$ where $Q_{t+1}(X) = 1 + \sum_{j=0}^t \alpha_j X Q_j(X)$ which is a polynomial of degree less than $t+1$. This concludes the proof of the induction claim.

We now prove the lower bound~\eqref{eqnlowerboundfixedsketch}. Let $Q_t \in \real_{t}[X]$ such that $Q_t(0)=1$ and $x_t - x^* = Q_t(H_S^{-1}H)(x_0-x^*)$. Multiplying both sides by $H^\frac{1}{2}$, we get $\Delta_t = H^\frac{1}{2} Q_t(H_S^{-1}H)(x_0-x^*)$. Note that $H^\frac{1}{2} (H_S^{-1}H)^k = (C_S^{-1})^k H^\frac{1}{2}$ for any $k \gre 0$, whence $\Delta_t = Q_t(C_S^{-1}) \Delta_0$. Consequently, $\delta_t = \frac{1}{2} (\Delta_0)^\top Q_t(C_S^{-1})^2 (\Delta_0)$. Using the eigenvalue decomposition $C_S^{-1} = \sum_{i=1}^d \lambda_i v_i v_i^\top$, we find that $\delta_t = \frac{1}{2}\sum_{i=1}^d Q_t(\lambda_i)^2 (v_i^\top \Delta_0)^2$, i.e., $\delta_t = \frac{1}{2} \sum_{i=1}^d Q_t(\lambda_i)^2 \xi_i^2$, and this yields the claim.

We now prove the optimality of PCG. It is well-known (e.g.,~\cite{d2018optimization}, Ch.~1, or~\cite{axelsson1996iterative}, Section~11.1.2) that the iterates $\{x_t\}$ of PCG starting at $x_0$ and the iterates $\{z_t\}$ of CG starting at $z_0 = H_S^{\frac{1}{2}} x_0$ and applied to $f_S : z \mapsto \frac{1}{2} \|A H_S^{-\frac{1}{2}} z\|_2^2 - b^\top H_S^{-\frac{1}{2}} z$ are related by the equation $x_t = H_S^{-\frac{1}{2}} z_t$ for all $t \gre 0$. Thus, we can leverage classical results of CG. (\textit{PCG is an instance of~\eqref{eqnfirstordermethods}}.) According to Theorem~5.3 in~\cite{nocedal2006numerical}, we have $z_{t+1} \in z_0 + \text{span}\!\left\{\nabla f_S(z_0), \hdots, \nabla f_S(z_{t})\right\}$ for all $t \gre 0$. Multiplying the latter inclusion by $H_S^{-\frac{1}{2}}$ and observing that $\nabla f_S(z_j) = H_S^{-\frac{1}{2}} \nabla f(x_j)$ for all $j \gre 0$, we obtain $x_{t+1} \in x_0 + H_S^{-1} \text{span}\!\left\{\nabla f(x_0), \hdots, \nabla f(x_{t})\right\}$. (\textit{PCG is optimal}.) According to Theorem~5.2 in~\cite{nocedal2006numerical}  (see, also,~\cite{hestenes1952methods, daniel1967conjugate}), we have $\|H^\frac{1}{2}H_S^{-\frac{1}{2}} (z_t - z^*)\|_2^2 = \min_{\substack{Q_t \in \real_t[X];\\Q_t(0)=1}} \|Q_t(H^\frac{1}{2}H_S^{-1}H^\frac{1}{2}) H^\frac{1}{2} H_S^{-\frac{1}{2}} (z_0-z^*)\|_2^2$. Using that $\Delta_t = H^\frac{1}{2} H_S^{-\frac{1}{2}} (z_t-z^*)$, the above equation becomes $\|\Delta_t\|_2^2 = \min_{\substack{Q_t \in \real_t[X];\\Q_t(0)=1}} \|Q_t(C_S^{-1}) \Delta_0\|_2^2$, which is equivalent to the claimed optimality result.

\subsubsection*{Proof of Theorem~\ref{theoremihsfixed}} 
The Polyak-IHS update with constant parameters $\mu$ and $\beta$ for $t \gre 1$ is equivalently written as $x_{t+1}-x^* = x_t-x^* - \mu H_S^{-1} H(x_t - x^*) + \beta (x_t - x_{t-1})$ for $t \gre 1$. Multiplying both sides of this equation by $H^\frac{1}{2}$, we obtain the recursion $\Delta_{t+1} = \left(I - \mu C_S^{-1} \right) \Delta_t + \beta (\Delta_t - \Delta_{t-1})$. This can be written as the linear dynamics
\begin{align} 
\label{eqndynamics}
    \begin{bmatrix} \Delta_{t+1}\\ \Delta_t \end{bmatrix} =  M_{\mu,\beta} \begin{bmatrix} \Delta_t \\ \Delta_{t-1} \end{bmatrix}\,,\quad \mbox{where} \quad M_{\mu, \beta} \defn \begin{bmatrix} (1+\beta)I_d - \mu C_S^{-1} & -\beta I_d \\ I_d & 0 \end{bmatrix}\,.
\end{align}
For the specific case $\beta = 0$, we obtain the IHS update and the associated recursion formula $\Delta_{t+1} = (I_d - \mu C_S^{-1} ) \Delta_t$. Using that $\delta_t = \frac{1}{2} \|\Delta_t\|^2$, we obtain for any $t \gre 0$ that $\delta_{t+1} \less  \|I_d - \mu C_S^{-1}\|_2^2 \cdot \delta_t$. The eigenvalues of $I_d -\mu C_S^{-1}$ are $1-\frac{\mu}{\lambda_i}$, whence $\|I_d - \mu C_S^{-1}\|_2 = \max\{|1-\frac{\mu}{\lambda_1}|, |1-\frac{\mu}{\lambda_d}|\}$. Conditional on $\mathcal{E}_\rho$, we have $(1-\sqrt{\rho})^2 \less \lambda_d \less \lambda_1 \less (1+\sqrt{\rho})^2$, whence $\|I_d - \mu C_S^{-1}\|_2 \less \max\!\left\{|1-\frac{\mu}{(1+\sqrt{\rho})^2}|, |1-\frac{\mu}{(1-\sqrt{\rho})^2}|\right\}$. Picking $\mu = \frac{(1-\rho)^2}{1+\rho}$ yields that $\|I_d - \mu C_S^{-1}\|^2_2 \less \frac{4 \rho}{(1+\rho)^2}$, and thus the claimed error bound~\eqref{eqnerrorihs}. For the Polyak IHS update ($\beta > 0)$, we obtain by induction from~\eqref{eqndynamics} that $\delta_{t+1} + \delta_t \less \|M_{\mu,\beta}^t\|^2_2 (\delta_1 + \delta_0)$ for any $t \gre 0$. The so-called Gelfand's formula states that $\lim  \|M_{\mu,\beta}^t\|^{\frac{2}{t}}_2 = \rho(M_{\mu,\beta})^2$ and consequently, $\limsup  \left(\frac{\delta_{t+1}+\delta_t}{\delta_1 + \delta_0}\right)^\frac{1}{t} \less \rho(M_{\mu,\beta})^2$, i.e., $\limsup \left(\frac{\delta_t}{\delta_0}\right)^\frac{1}{t} \less \rho(M_{\mu,\beta})^2$. It remains to bound the spectral radius $\rho(M_{\mu,\beta})^2$. From standard analysis of the heavy-ball method~\cite{polyak1964some}, we have for $\beta \gre \max \left\{ |1-\sqrt{\frac{\mu}{\lambda_1}}|, |1-\sqrt{\frac{\mu}{\lambda_d}}| \right\}^2$ that $\rho(M_{\mu,\beta}) \less \sqrt{\beta}$. Choosing $\mu = (1-\rho)^2$ and $\beta = \rho$ yields the claimed result, i.e., $\rho(M_{\mu,\beta})^2 \less \rho$.

\subsubsection*{Proof of Theorem~\ref{thmexacterrorfixed}}
Fix $t \gre 0$. Using the error recursion formula $\Delta_{t+1} = (I_d - \mu C_S^{-1}) \Delta_t$, we obtain by induction that $\Delta_t = (I_d - \mu C_S^{-1})^t \Delta_0$. Using the eigenvalue decomposition $C_S = \sum_{i=1}^d \lambda_i v_i v_i^\top$, we get $(I_d - \mu C_S^{-1})^t = \sum_{i=1}^d (1-\mu \lambda_i^{-1})^t v_i v_i^\top$, whence $\delta_t = \frac{1}{2} \sum_{i=1}^d (1-\mu \lambda_i^{-1})^{2t} (v_i^\top \Delta_0)^2$. For Gaussian and Haar embeddings, it holds that each eigenvector $v_i$ is independent of $\lambda_i$ and that $\mathbb{E}\{v_i v_i^\top\} = \frac{1}{d} I_d$. Therefore, we obtain that $\mathbb{E}\{\delta_t\} = \mathbb{E}\{\frac{1}{d} \sum_{i=1}^d (1-\mu \lambda_i^{-1})^{2t}\} \mathbb{E}\{\delta_0\}$, which is the claimed result.

\subsubsection*{Proof of Theorem~\ref{theoremmp}}

We introduce 
\begin{align}
    \bar{\varphi}_t(\gamma) = \frac{1}{\left(\gamma + \kappa\right)^{2t}} \int_{-1}^1 (z-\gamma)^{2t} \sqrt{1-z^2} \,h(z) \mathrm{d}z\,,
\end{align}
where $a \defn \lambda_+^{-1}$, $b \defn \lambda_-^{-1}$, $\kappa \defn \frac{b+a}{b-a}$ and $h(z) \defn \frac{1}{\left(z + \kappa \right)^2}$. With the change of variable $\mu = (\frac{b-a}{2}\gamma + \frac{a+b}{2})^{-1}$, we get that $\bar{\varphi}_t(\gamma) = \Gamma_t^\rho(\mu)$. Note that $r_\infty = \left(\frac{b-a}{b+a}\right)^2$, $\mu^* = \frac{2}{a+b}$, and $\gamma \in (-1,1)$ if and only if $\mu \in (b^{-1}, a^{-1})$. We leverage the next result, whose proof is essentially based on Laplace approximations of integrals~\cite{de1981asymptotic}.
\begin{lemma}
	\label{lemmaasympbarphi}
	For $\gamma \in (-1,1)$ such that $\gamma \neq 0$, it holds that $\lim_{t \to +\infty} \frac{\bar{\varphi}_t(\gamma)}{\bar{\varphi}_t(0)} = +\infty$ and $\lim_{t \to +\infty} \frac{\bar{\varphi}_{t+1}(0)}{\bar{\varphi}_{t}(0)} = \left(\frac{b-a}{b+a}\right)^2$.
\end{lemma}
From Lemma~\ref{lemmaasympbarphi}, it follows that $\lim_{t \to +\infty}\frac{\Gamma^\rho_t(\mu)}{\Gamma^\rho_t(\mu^*)} = +\infty$ for any $\mu \in (b^{-1}, a^{-1})$ such that $\mu \neq \mu^*$, and that $\lim_{t \to +\infty}\frac{\Gamma^\rho_{t+1}(\mu^*)}{\Gamma^\rho_t(\mu^*)} = \left(\frac{b-a}{b+a}\right)^2$. It remains to show that the divergence result holds for any $\mu \in \real$, which follows from the convexity of $\Gamma^\rho_t$. Indeed, fix any $\mu \in \real$, and let $\varepsilon > 0$ be sufficiently small such that $\mu_\varepsilon = \varepsilon \mu + (1-\varepsilon) \mu^* \in (b^{-1}, a^{-1})$. Then, by convexity of $\Gamma^\rho_t$, we have the inequality $\liminf_{t \to +\infty}\frac{\Gamma^\rho_t(\mu_\varepsilon)}{\Gamma^\rho_t(\mu^*)} \less \varepsilon \liminf_{t \to +\infty}\frac{\Gamma^\rho_t(\mu)}{\Gamma^\rho_t(\mu^*)} + (1-\varepsilon)$. The latter left-hand side is equal to $+\infty$, and this concludes the proof.

\subsubsection*{Proof of Lemma~\ref{lemmaasympbarphi}}

We note that $\bar{\varphi}_t(\gamma) = \frac{1}{(\gamma + \kappa)^{2t}} \psi_t(\gamma)$ for any $\gamma \in (-1,1)$, where $\psi_t(\gamma) \defn \int_{-1}^{1} (z-\gamma)^{2t} \sqrt{1-z^2}\,h(z)\,\mathrm{d}z$. For $\gamma \in (-1,1)$ such that $\gamma \neq 0$, we claim that
\begin{align}
\label{eqnlaplaceapproximations}
	\psi_t(\gamma) \sim \sqrt{\frac{\pi}{8 e}} \frac{(1+|\gamma|)^{2t+3/2}}{(-\gamma/|\gamma|+\kappa)^2\, t^\frac{3}{2}}\,,\quad \mbox{and} \quad \psi_t(0) \sim \sqrt{\frac{\pi}{2e}} \frac{1+\kappa^2}{(\kappa^2-1)^2} \frac{1}{t^\frac{3}{2}}\,.
\end{align}
Given the above approximations, it follows that $\bar{\varphi}_t(\gamma) \sim \frac{1}{(-\gamma/|\gamma|+\kappa)^2} \sqrt{\frac{\pi}{8 e}} \frac{(1+|\gamma|)^{2t+3/2}}{(\gamma + \kappa)^{2t}\, t^\frac{3}{2}}$ and $\bar{\varphi}_t(0) \sim \sqrt{\frac{\pi}{2e}} \frac{1+\kappa^2}{(\kappa^2-1)^2} \frac{1}{t^\frac{3}{2} \kappa^{2t}}$. Using that $\frac{1+|\gamma|}{1+\frac{\gamma}{\kappa}} > 1$, it follows that $\lim_{t \to +\infty} \frac{\bar{\varphi}_t(\gamma)}{\bar{\varphi}_t(0)} = +\infty$. On the other hand, we have $\frac{\bar{\varphi}_{t+1}(0)}{\bar{\varphi}_t(0)} \underset{t \to +\infty}{\sim} \frac{t^\frac{3}{2}}{(t+1)^\frac{3}{2}} \kappa^{-2}$, whence $\lim_{t \to +\infty} \frac{\bar{\varphi}_{t+1}(0)}{\bar{\varphi}_t(0)} = \kappa^{-2} = \left(\frac{b-a}{b+a}\right)^2$. It remains to show the asymptotic equivalences~\eqref{eqnlaplaceapproximations}. Let $\gamma \in (-1,1)$. We distinguish three cases, that is, $\gamma < 0$, $\gamma > 0$ and $\gamma=0$.

\textbf{Case 1}: $-1 < \gamma < 0$. For any $z \in [-1,\gamma]$, we have $|z-\gamma| \less |\gamma+1| < 1$, it follows, by dominated convergence, that the integral $\lim_{t \to \infty} \int_{-1}^{\gamma} (z-\gamma)^{2t} \sqrt{1-z^2}\,h(z)\,\mathrm{d}z = 0$. Thus, it suffices to characterize the asymptotic behavior of the integral $\int_{\gamma}^1 (z-\gamma)^{2t} \sqrt{1-z^2}\,h(z)\,\mathrm{d}z = \int_{\gamma}^1 e^{t g_t(z)}\,h(z)\, \mathrm{d}z$, where we introduced $g_t(z) = 2 \log(z-\gamma) + \frac{1}{2t}\log(1-z^2)$. We have that: (i) the function $g_t$ is continuous over $(\gamma, 1)$, (ii) $\lim_{z\to 1^-}g_t(z) = -\infty$, and (iii) $\lim_{z\to \gamma^+}g_t(z) = -\infty$. Therefore, the function $g_t$ admits a maximizer $z^* \in (\gamma, 1)$. By first-order optimality conditions, we have $g_t^\prime(z^*) = 0$ which, after re-arranging, yields $(1+2t) {z^*}^2 - \gamma z^* - 2t = 0$, whose two solutions are equal to $z_\pm = \frac{\gamma}{2(1+2t)} \pm \sqrt{\frac{2t}{2t+1}+\frac{\gamma^2}{4(1+2t)^2}}$. The solution $z^*$ corresponds to the positive branch. Indeed, when $t \to +\infty$, we have that $z_- \to -1$, whereas we must have $z^* > \gamma > -1$, hence, ruling out the equality $z^* = z_-$. Thus, the maximizer $z^*$ is unique, equal to $z_+$, and, by Laplace approximation of integrals~\cite{de1981asymptotic}, we have $\psi_t(\gamma) = e^{t g_t(z^*)} h(z^*) \sqrt{\frac{2\pi}{t |g^{\prime \prime}_t(z^*)|}} \left(1 + \mathcal{O}\left(\frac{1}{t}\right)\right)$. Expanding $z^*$ at first-order in terms of $1/t$, we obtain after straightforward though tedious calculations that $\psi_t(\gamma) \sim \sqrt{\frac{\pi}{8e}} \frac{(1-\gamma)^{2t + 3/2}}{(1+\kappa)^2\,t^{\frac{3}{2}}}$, which is the claimed result.

\textbf{Case 2}: $0 < \gamma < 1$. With change of variable $z^\prime = -z$ and setting $\gamma^\prime = -\gamma$, this case study follows exactly the same lines as the previous one, and we obtain $\psi_t(\gamma) \sim \sqrt{\frac{\pi}{8e}} \frac{(1+\gamma)^{2t + 3/2}}{(-1+\kappa)^2\,t^\frac{3}{2}}$.

\textbf{Case 3}: $\gamma=0$.  Separating the integral defining $\psi_t(0)$ at $z=0$, we need to study separately the asymptotics of the two following integrals: $\int_{-1}^0 z^{2t} \sqrt{1-z^2} h(z)\,\mathrm{d}z$ and $\int_{0}^1 z^{2t} \sqrt{1-z^2} h(z)\,\mathrm{d}z$. Following similar steps as in the first case-study, we obtain that $\int_{-1}^0 z^{2t} \sqrt{1-z^2} h(z)\,\mathrm{d}z \sim \sqrt{\frac{\pi}{8e}} \frac{1}{(1+\kappa)^2 \, t^\frac{3}{2}}$ and $\int_{0}^1 z^{2t} \sqrt{1-z^2} h(z)\,\mathrm{d}z \sim \sqrt{\frac{\pi}{8e}} \frac{1}{(-1+\kappa)^2 \, t^\frac{3}{2}}$. By summing the two above expansions, we obtain the claimed result.

\section{Pre-conditioned First-order Methods with Refreshed Embeddings}
\label{sectionihsrefreshed}

In contrast to the extreme eigenvalues analysis for a fixed embedding, our analysis with refreshed embeddings provides error guarantees which are fully predictable in terms of the inverse moments~\eqref{eqninversemoments} of the matrix $C_S$. We provide next an exact error formula for the IHS with refreshed embeddings $\{S_t\}$ that satisfy Condition~\ref{condition:unbiasedmoments}. We emphasize that after the appearance of a preliminary version of this work, the follow-up work~\cite{lacotte2020limiting} extended the next result to the specific case of Haar embeddings and the SRHT in the asymptotic setting $n,d,m \to \infty$. 
\begin{theorem}
\label{theoremexacterrorrefreshed}
Suppose that the i.i.d.~random embeddings $S_t \in \real^{m \times n}$ satisfy Condition~\ref{condition:unbiasedmoments}. Then, the iterates $\{x_t\}$ of the IHS with step sizes $\{\mu_t\}$ (independent of $\{S_t\}$) satisfy the exact error formula 
\begin{align}
\label{eqnexacterrorrefreshed}
    \mathbb{E}\{\delta_t\} = \prod_{j=0}^{t-1} \Big[ \big (\,\frac{\moment_1}{ \sqrt{\moment_2}}-\mu_j\, \sqrt{\moment_2} \big)^2 + 1-\frac{\moment_1^2}{\moment_2} \Big] \mathbb{E}\{\delta_0\}\,.
\end{align}
Consequently, the minimal error is obtained with the constant step size $\mu_t \equiv \mu = \theta_1 / \theta_2$ and we have
\begin{align}
\label{eqnexactoptimalerrorrefreshed}
	\frac{\mathbb{E}\{\delta_t\}}{\mathbb{E}\{\delta_0\}} = \left( 1-\frac{\moment_1^2}{\moment_2} \right)^t\,.
\end{align}
\end{theorem}
An important aspect of the results above is that the expected error formula is exact, and it holds universally for every input data $(A,b)$ with equality. In particular, one cannot hope to do better by adjusting the step-sizes as long as they are independent of the randomness in the sketches. The required Condition~\ref{condition:unbiasedmoments} on the inverse moments for Theorem~\ref{theoremexacterrorrefreshed} contrasts to earlier works on preconditioning, where $S$ is required to satisfy subspace embeddings properties for convergence results to hold. An exception is the method proposed in~\cite{cohen2014preconditioning}, where the authors obtain upper-bounds for leverage score based on row sampling by directly analyzing the expectation of the error iteration. 

From Corollary~3.5 in~\cite{knyazev2008steepest}, we get that FCG with $0 < k_{t+1} \less k_t + 1$ and $k_t \less t$ improves locally on the IHS update, i.e.,
\begin{align}
\label{eqnfcglocalimprovementihs}
    \mathbb{E}\{\delta_{t+1}\} \less \min_\mu \frac{1}{2} \mathbb{E}\{\|x_t - \mu H_S^\dagger \nabla f(x_t)\|_H^2\}\,,
\end{align}
Using the exact error formula~\eqref{eqnexactoptimalerrorrefreshed} of the IHS and the local improvement inequality~\eqref{eqnfcglocalimprovementihs}, we obtain the following result.
\begin{corollary}
\label{corollaryupperboundfcg}
Suppose that the i.i.d.~random embeddings $S_t \in \real^{m \times n}$ satisfy Condition~\ref{condition:unbiasedmoments}. Then, the iterates $\{x_t\}$ of FCG with $0 < k_{t+1} \less k_t + 1$ and $k_t \less t$ satisfy
\begin{align}
\label{eqnupperboundfcg}
    \frac{\mathbb{E}\{\delta_t\}}{\mathbb{E}\{\delta_0\}} \less \left(1-\frac{\theta_1^2}{\theta_2}\right)^t\,.
\end{align}
\end{corollary}
Using the formulas~\eqref{eqngaussianmoments} of $\theta_1$ and $\theta_2$, we obtain for refreshed Gaussian embeddings with $m \gre d+4$ that $1-\theta_1^2/\theta_2 = \frac{d}{m} (1+o(1))$. Note that this scales as the high-probability upper bound~\eqref{eqnerrorpcggaussian} on the error of PCG based on extreme eigenvalues analysis. Using the formulas~\eqref{eqnsrhtmoments} of $\theta_1$ and $\theta_2$, we obtain for Haar embeddings the finite-sample approximation $1-\theta_1^2/\theta_2 \approx \frac{d}{m}\frac{(n-m)m}{d^2 + nm - 2dm}$. Note that this is always smaller than the rate $d/m$ of Gaussian embeddings. How does this compare to the rate of PCG? After the appearance of a preliminary version of the present work, the follow-up work~\cite{lacotte2020optimal} carried out an exact extreme eigenvalue analysis of preconditioned first-order methods with a fixed embedding in the asymptotic setting $m,n,d \to +\infty$. It is shown that the rate of Polyak IHS with a fixed embedding (and thus of PCG) scales as $1-\theta_1^2/\theta_2$ for Haar matrices. Hence, based on our analysis so far, it is unclear whether FCG with refreshed Gaussian or Haar embeddings outperforms PCG. 

Corollary~3.5 in~\cite{knyazev2008steepest} actually provides a bound better than~\eqref{eqnfcglocalimprovementihs}: in fact, it is shown that FCG improves locally on the Polyak IHS update, that is,
\begin{align}
\label{eqnfcglocalimprovementpolyakihs}
    \mathbb{E}\{\delta_{t+1}\} \less \min_{\mu, \beta} \frac{1}{2} \mathbb{E}\{\|x_t - \mu H_S^\dagger \nabla f(x_t) - \beta (x_t - x_{t-1})\|_H^2\}\,.
\end{align}
It is therefore of interest to characterize the convergence rate of the optimal Polyak-IHS with refreshed embeddings, to see whether we can improve the upper bound~\eqref{eqnupperboundfcg} on FCG. Surprisingly, in contrast to standard first-order methods, heavy-ball momentum does not accelerate the IHS with refreshed embeddings.
\begin{theorem}
\label{theoremmomentumrefreshed}
Let $S_t$ be a sequence of i.i.d.~random embeddings that satisfy Condition~\ref{condition:unbiasedmoments}. Given step size and momentum parameters $\mu, \beta$, there exist initial points $x_0,x_1 \in \real^d$ such that the error of Polyak IHS is lower bounded as
\begin{align}
\label{eqnmomentumrefreshed}
    \limsup_{t \to \infty} \left(\frac{\mathbb{E}\{\delta_t\}}{\mathbb{E}\{\delta_0\}}\right)^\frac{1}{t} \gre 1-\frac{\theta_1^2}{\theta_2} \,.
\end{align}
\end{theorem}
The result of Theorem~\ref{theoremmomentumrefreshed} is asymptotic, i.e., it holds for $t\to \infty$, and this is essentially due to the approximation of the spectral radius by Gelfand's formula (see, e.g.,~\cite{kozyakin2009}). In contrast, FCG terminates within $d-1$ steps (with exact arithmetic). Despite this, the lower bound~\eqref{eqnmomentumrefreshed} remains informative about the empirical behavior of the Polyak-IHS with refreshed embeddings as we will clearly observe in numerical simulations in Section~\ref{sectionnumericalsimulations}.

\subsection{Proofs of Results in Section~\ref{sectionihsrefreshed}}


\subsubsection*{Proof of Theorem~\ref{theoremexacterrorrefreshed}}

Multiplying both sides of the IHS update~\eqref{eqnihs} by $H^\frac{1}{2}$, subtracting $H^\frac{1}{2} x^*$ and using that $H x^* = b$, we obtain the recursion formula $\Delta_{t+1} = (I_d - \mu_t C_{S_t}^{-1}) \Delta_t$. Taking the squared norm, it follows that $\|\Delta_{t+1}\|_2^2 = \Delta_t^\top (I_d - \mu_t C_{S_t}^{-1})^2 \Delta_t$. Taking expectations, using independence of the embeddings and using Condition~\ref{condition:unbiasedmoments}, we obtain the error recursion formula $\mathbb{E}\{\delta_{t+1}\} = (1-2 \mu_t \theta_1 + \mu_t^2 \theta_2) \cdot \mathbb{E}\{\delta_t\}$. It follows by induction that $\mathbb{E}\{\delta_t\} = \prod_{j=0}^{t-1} (1-2\mu_t \theta_1 + \mu_t^2 \theta_2) \mathbb{E}\{\delta_0\}$. Observing that $1-2\mu_t \theta_1 + \mu_t^2 \theta_2 = (\frac{\theta_1}{\sqrt{\theta_2}} - \mu_t \sqrt{\theta_2})^2 + 1-\frac{\theta_1^2}{\theta_2}$, we obtain that the minimal expected error is attained for $\mu_t = \theta_1 / \theta_2$, and the convergence rate is equal to $1-\theta_1^2/\theta_2$, which is the claimed result.

\subsubsection*{Proof of Theorem~\ref{theoremmomentumrefreshed}}

Multiplying both sides of the Polyak IHS update~\eqref{eqnpolyakihs} by $H^\frac{1}{2}$,  subtracting $H^\frac{1}{2} x^*$ and using that $H x^* = b$, we obtain the recursion formula $\Delta_{t+1} = \Delta_t - \mu C_{S_t}^{-1} \Delta_t + \beta(\Delta_t-\Delta_{t-1})$. For $t \gre 0$, setting $a_t \defn \delta_{t+1}$, $b_t \defn -\frac{\beta}{2} \mathbb{E}[\Delta_{t}^\top \Delta_{t+1}]$ and $c_t \defn \beta^2 \delta_t$, we obtain from the latter recursion formula after simple calculations that
\begin{align}
\label{EqnLDSMomentum}
X_{t+1}
= \begin{bmatrix} \eta & 2 \gamma & 1 \\ - \beta \gamma & - \beta & 0 \\ \beta^2 & 0 & 0 \end{bmatrix}
X_t\,,
\end{align}
where we introduced the notations $X_t = [a_t, b_t, c_t]^\top$, $\eta \defn (1+\beta)^2 - 2\mu \theta_1 (1+\beta) + \mu^2 \theta_2$ and $\gamma \defn 1 + \beta - \mu \theta_1$. We denote the characteristic polynomial of this time-invariant linear dynamical system by $\chi_{\mu, \beta}(X)$ and its root radius (i.e., the largest module of its roots) by $\Lambda(\mu,\beta)$. Given $\mu, \beta > 0$, it follows from~\eqref{EqnLDSMomentum} and fundamental results in time-invariant linear dynamical systems that there exist initializations of $x_0$ and $x_1$ such that $\limsup \big(\frac{\|X_t\|_2}{\|X_0\|_2}\big)^{\frac{1}{t}} \gre \Lambda(\mu,\beta)$. This further implies that
\begin{align}
    \limsup_{t \to \infty} \left(\frac{\mathbb{E}\{\delta_t\}}{\mathbb{E}\{\delta_0\}}\right)^{\frac{1}{t}} \gre \Lambda(\mu,\beta)\,.
\end{align}
The next result concludes the present analysis. Its proof is fairly involved and we defer it to Section~\ref{sectionminimalrootradius}.
\begin{theorem}[Minimal root radius]
\label{theoremminimalrootradius}
It holds that
\begin{align}
    \label{eqnminimalrootradius}
    \inf_{\mu, \beta \gre 0} \Lambda(\mu, \beta) = 1 - \frac{\theta_1^2}{\theta_2}\,.
\end{align}
\end{theorem}

\section{Numerical Comparisons}
\label{sectionnumericalsimulations}

Differently from earlier works on variable preconditioning~\cite{knyazev2008steepest, golub1999inexact, axelsson1991black, notay2000flexible} and their worst-case approximation error analysis, we aim to investigate whether randomized and independent variable preconditioners may improve the performance over a fixed linear operator, thanks to additional randomness. Strikingly, the answer seems negative. We compare numerically the several methods explored so far. Importantly, we observe that PCG with a fixed SRHT-based preconditioner yields the best numerical performance.

We have established for PCG with a fixed preconditioner $H_S$ that, conditional on $\mathcal{E}_\rho$, the error satisfies the upper bound $\delta_t / \delta_0 \less 4 \rho^t$. On the other hand, we have established for FCG with refreshed embeddings that the error satisfies the upper bound $\mathbb{E}\{\delta_t\} / \mathbb{E}\{\delta_0\} \less (1-\theta_1^2/\theta_2)^t$. It is of interest to investigate numerically whether these upper bounds are tight or overly conservative. On the other hand, with refreshed embeddings, we expect to observe empirically the exact error formula $\mathbb{E}\{\delta_t\}/\mathbb{E}\{\delta_0\} = (1-\theta_1^2/\theta_2)^t$ for the IHS with step size $\mu = \theta_1/\theta_2$, as well as a worse performance of the Polyak-IHS for any $\beta > 0$ (and say $\mu = \theta_1/\theta_2$ for simple comparison with the IHS). In particular, we consider several truncation thresholds $k_t$ for FCG. We compute FCG with full orthogonalization $k_t = t-1$, that is, GCC. We also consider FCG with $p_t = 1$. According to~\cite{knyazev2008steepest} (see Section 7 therein), this is equivalent to the inexact preconditioned conjugate gradient method (IPCG) proposed by~\cite{golub1999inexact}. We repeat all these experiments for both Gaussian embeddings (Figure~\ref{figurecomparisongaussian}), Haar matrices (Figure~\ref{figurecomparisonhaar}) as well as the SRHT (Figure~\ref{figurecomparisonsrht}). 

Note that the SRHT does not satisfy Condition~\ref{condition:unbiasedmoments} and it is unknown whether the exact error formula~\eqref{eqnexacterrorrefreshed} of the IHS with refreshed SRHT holds. However, we use heuristically the step size $\mu = \mbox{tr}\,\mathbb{E}\{C_S^{-1}\} / \mbox{tr}\, \mathbb{E}\{C_S^{-2}\}$ based on the finite-sample approximations~\eqref{eqnsrhttrace} of these trace formula. We expect the IHS with refreshed SRHT to have similar performance as with Haar embeddings, since the trace formula and the limiting spectral distributions of the matrices $C_S$ for both the SRHT and Haar embeddings are asymptotically equal (see~\cite{dobriban2019asymptotics} for the latter fact).

Numerical results are consistent with our theoretical predictions. The theoretical bound $(1-\theta_1^2/\theta_2)^t$ predicts with high accuracy the relative error $\delta_t/\delta_0$ for the IHS with any of the aforementioned embeddings. The performance of GCC improves on the performance of IPCG, which itself improves on the IHS. Finally, we observe that heavy-ball momentum does not improve the IHS. As expected (and already observed in~\cite{lacotte2020limiting}), the numerical performance of the IHS with refreshed SRHT is close to that with Haar matrices.

\begin{figure}[h!]
	\centering
	\includegraphics[width=\textwidth]{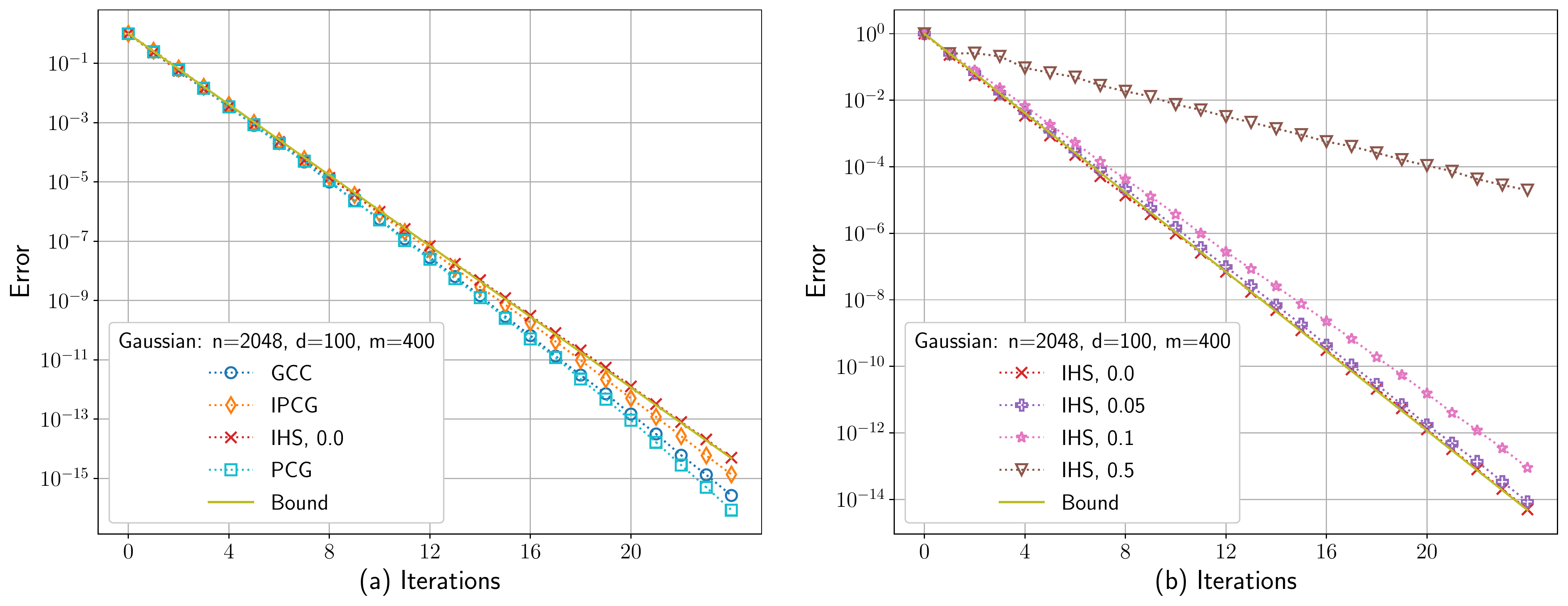}
	\caption{Comparison of the relative error $\delta_t/\delta_0$ for GCC, IPCG, IHS and Polyak IHS with $\beta \in \{0.05, 0.1, 0.5\}$ with refreshed Gaussian embeddings, and PCG with a fixed Gaussian embedding. The curve 'Bound' is the theoretical bound $(1-\theta_1^2/\theta_2)^t$.}
	\label{figurecomparisongaussian}
\end{figure}

\begin{figure}[h!]
	\centering
	\includegraphics[width=\textwidth]{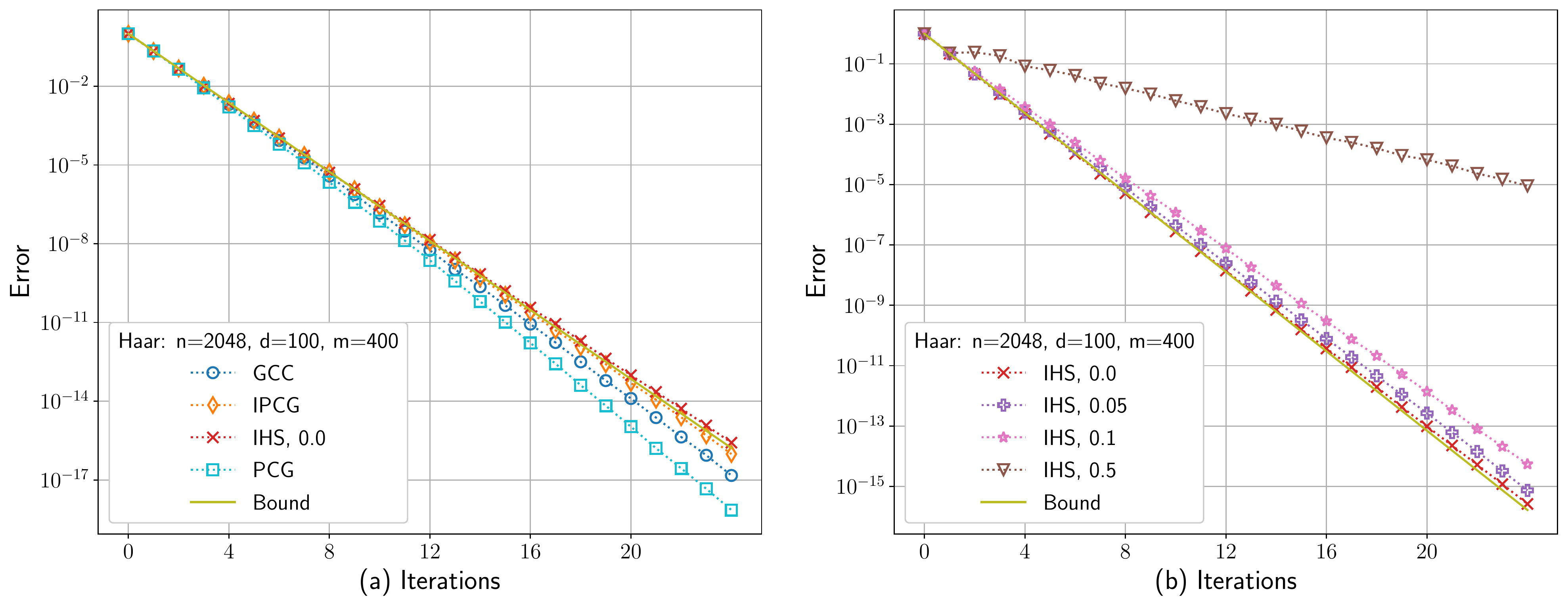}
	\caption{Comparison of the relative error with Haar matrices.}
	\label{figurecomparisonhaar}
\end{figure}

\begin{figure}[h!]
	\centering
	\includegraphics[width=\textwidth]{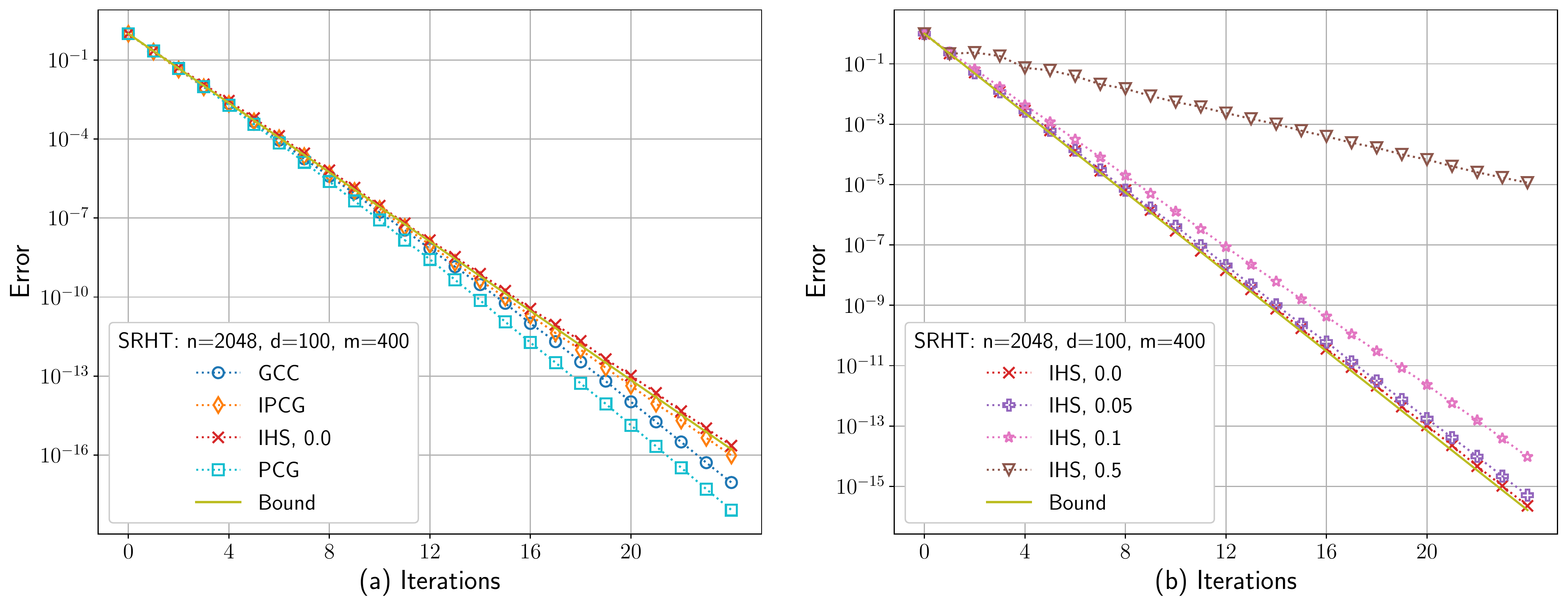}
	\caption{Comparison of the relative error with the SRHT.}
	\label{figurecomparisonsrht}
\end{figure}

\section{Optimized Sketch Size and Computational Complexity}

We say that an approximate solution $\wtilde x$ is $(\varepsilon,\delta)$-accurate if its relative error satisfies $\delta_t / \delta_0 \less \varepsilon$ for any $t \gre 0$ with probability $1-\delta$. Recall from Table~\ref{tableembeddingproperties} that we define the critical sketch size $m_\delta$ of an embedding $S$ as $m_\delta = \inf\{k \mid \mathbb{P}(\mathcal{E}_\rho^m) \gre 1-\delta\,,\forall m \gre k/\rho\,,\,\,\forall \rho \in (0,1)\}$. Our analysis and numerical experiments so far suggest that PCG yields the best performance among the different methods we have considered. According to~\eqref{eqnupperboundpcgrate}, the number of iterations $T$ required for PCG to reach an $(\varepsilon,\delta)$-accurate solution is given by $T = \lceil\log(4/\varepsilon) / \log(1/\rho)\rceil$, provided that $m \gre m_\delta / \rho$. Each iteration of PCG involves matrix-vector multiplications with $A$ and $A^\top$ with cost $\mathcal{O}(nd)$, as well as solving a linear system involving $H_S$. Given a precomputed matrix factorization of $H_S$ (e.g. QR decomposition of $SA$ or Cholesky decomposition of $H_S$), the cost of solving this linear system is $\mathcal{O}(d^2)$, which is negligible compared to $\mathcal{O}(nd)$. For $m \gre d$, the cost of the matrix decomposition is $\mathcal{O}(md^2)$. Denoting by $\mathcal{C}_{S,n,d}(m)$ the cost of forming the sketch $S \cdot A$, we obtain the total complexity $\mathcal{C}(m) = \mathcal{C}_{S,n,d}(m) + \mathcal{O}(md^2) + \mathcal{O}\!\big(nd \frac{\log(1/\varepsilon)}{\log(m/m_\delta)}\big)$, which we aim to minimize over $m > m_\delta$. If the factorization cost $\mathcal{O}(m d^2)$ dominates uniformly over $m > m_\delta$ the sketch and iteration costs, then the optimal solution is $m = \mathcal{O}(m_\delta)$. In order to avoid this trivial case for the embeddings we consider in this work, we assume that the problem is highly overparameterized, that is, $n > d^2$.

Our optimized sketch size will essentially improve the logarithmic terms involved in the total cost, and these are sensitive to numerical constants hidden in the asymptotic notations. Nonetheless, we choose to discard them in our analysis. This will considerably simplify the discussion as these numerical constants may depend on specific implementations choices, e.g., factorization of $H_S$, structured matrix-vector multiplications, as well as hardware specifications. Despite this choice, we will see that our analysis is fairly consistent with our empirical observations. As for future work, it may be of great practical interest to extend the present approach to a hardware specific setting, in order to optimize the total CPU-time of PCG as opposed to the flops count in an idealized RAM model. 

To our knowledge, we are the first to investigate formally the optimal sketch size for solving linear systems with sketching-based preconditioned iterative solvers. The authors of~\cite{meng2014lsrn} provide a short empirical investigation on the effect of the oversampling factor $\gamma = m/m_\delta$ on the overall running time for values $\gamma=\mathcal{O}(1)$. They do observe sensitive variations and an optimal trade-off between the different costs, and they prescribe then to use the fixed value $\gamma=2$. Similar choices are prescribed in~\cite{rokhlin2008fast, avron2010blendenpik} based on empirical observations.

\subsection{Optimized Complexity for the SRHT}
\label{sectionoptimalcomplexitysrht}

For $\delta = \mathcal{O}(1/d)$, the critical sketch size for the SRHT (see Table~\ref{tableembeddingproperties}) is $m_\delta \asymp d \log d$, and the sketching cost is $\mathcal{C}_{S,n,d}(m) = nd \log m$. We thus aim to minimize the cost function $\mathcal{C}(m) = nd \log m + md^2 + nd \frac{\log(1/\varepsilon)}{\log(m/(d\log d))}$. The next result provides an optimized sketch size and cost. We leave the proof to the reader, as it involves plugging-in the claimed value for the optimal sketch size into $\mathcal{C}(m)$ and then simple calculations.
\begin{theorem}
\label{theoremoptimalcomplexitysrht}
Given $\varepsilon > 0$ and assuming that $n > d^2$, we have the following total computational cost of PCG with a fixed SRHT to reach an $(\varepsilon, \mathcal{O}(1/d))$-accurate solution.
\begin{enumerate}[(i)]
    \item Suppose that $\sqrt{\log(1/\varepsilon)} < \log(\frac{n}{d^2})$. Then, for $m^* \asymp e^{\sqrt{\log(1/\varepsilon)}} d \log d$, we have the optimized complexity
    \begin{align}
        \mathcal{C}(m^*) = \mathcal{O}\!\left(nd \, (\log d + \sqrt{\log(1/\varepsilon)})\right)\,.
    \end{align}
    \item Suppose that $\sqrt{\log(1/\varepsilon)} \gre \log(\frac{n}{d^2})$. Then, for $m^* \asymp \frac{n}{d} \cdot\max\{\log d, \frac{\log(1/\varepsilon)}{\log(n/d^2)}\}$, we have the optimized complexity
    \begin{align}
        \mathcal{C}(m^*) = \mathcal{O}\!\left(nd \, \big( \log d + \frac{\log(1/\varepsilon)}{\log(n/d^2)}\big) \right)\,.
\end{align}
\end{enumerate}
\end{theorem}
Our optimized sketch size $m^*$ clearly improves on the classical choice $m_\text{cl} = \mathcal{O}(m_\delta)$ prescribed in the literature (e.g.~\cite{rokhlin2008fast, meng2014lsrn, avron2010blendenpik})  which yields the total cost $\mathcal{C}_\text{cl} = \mathcal{O}(nd (\log d + \log(1/\varepsilon)))$. For instance with $n=10^7$, $d=50$ and the standard statistical precision $\varepsilon=d/n$, we obtain (after discarding numerical constants within asymptotic notations) that $\frac{\mathcal{C}(m^*)}{\mathcal{C}_\text{cl}} \approx 0.46$.

\subsection{Optimized Complexity for Gaussian Embeddings}
\label{SectionTradeOffGaussian}

For $\delta = e^{-\mathcal{O}(d)}$, the critical sketch size for Gaussian embeddings (see Table~\ref{tableembeddingproperties}) is $m_\delta \asymp d$. We thus aim to minimize the cost function $\mathcal{C}(m) = \mathcal{C}_{S,n,d}(m) + md^2 + nd \frac{\log(1/\varepsilon)}{\log(m/d)}$. The worst-case sketching cost is $\mathcal{O}(ndm)$ which is then the bottleneck for any sketch size $m \gre d$, and it is in fact as large as the cost of a direct method to solve~\eqref{eqnconvexquadratic} exactly. As argued in~\cite{meng2014lsrn} (see Section 4.4 therein), the sketching time of Gaussian embeddings is usually much smaller in practice; in contrast to the SRHT, sketching is easily parallelized over the rows of $S$. For the sake of simplicity, we make the idealized assumption that the sketch $S \cdot A$ can be formed in time $\mathcal{O}(nd)$ through parallelization and we ignore the communication costs among threads. The sketching cost $C_{S,n,d}(m) = nd$ is then negligible compared to the other terms and we have $\mathcal{C}(m) = md^2 + nd \frac{\log(1/\varepsilon)}{\log(m/d)}$.

Given a real positive number $a$, we recall that the Lambert function $W(a)$ is defined as the unique real solution to the equation $W(a)e^{W(a)} = a$, and we have the bound $\exp(W(a)) \less \frac{2a+1}{1+\log(a+1)}$ (see, e.g., Theorem 2.3 in~\cite{hoorfar2008inequalities}). The cost function $\mathcal{C}(m)$ can be written in terms of the ratio $\alpha = m/d$ as $\alpha \mapsto d^3 (\alpha + \frac{n/d^2 \log(1/\varepsilon)}{\log \alpha})$ whose minimizer $\alpha^*$ satisfies the equation $\alpha^* \log \alpha^* = \frac{n}{d^2} \log(1/\varepsilon)$. Hence, we have $\alpha^* = \exp(W(\frac{n}{d^2} \log(1/\varepsilon)))$, i.e., $m^* = d \cdot \exp(W(\frac{n}{d^2} \log(1/\varepsilon)))$ and $\mathcal{C}(m^*) = 2d^3 \cdot \exp(W(\frac{n}{d^2} \log(1/\varepsilon)))$. Using the aforementioned bound on the Lambert function, we immediately obtain the following result.
\begin{theorem}
\label{theoremoptimalcomplexitygaussian}
Given $\varepsilon > 0$ and assuming that $n > d^2$, the total computational cost of PCG with a fixed Gaussian embedding to reach an $(\varepsilon, e^{-\mathcal{O}(d)})$-accurate solution with $m^* = d \cdot \exp(W_0(\frac{n}{d^2} \log(1/\varepsilon)))$ satisfies
\begin{align}
\label{eqnoptimizedcostgaussian}
    \mathcal{C}(m^*) = \mathcal{O}\!\left(nd \cdot \frac{\log(1/\varepsilon)}{\log(n/d^2)}\right)\,.
\end{align}
\end{theorem}
In contrast, the classical choice $m=\mathcal{O}(d)$ yields the total cost $\mathcal{C}_\text{cl} = \mathcal{O}(nd \log(1/\varepsilon))$. For instance, with $n=10^7$, $d=50$ and the standard statistical precision $\varepsilon=d/n$, we obtain (after discarding numerical constants within asymptotic notations) that $\frac{\mathcal{C}(m^*)}{\mathcal{C}_\text{cl}} \approx 0.12$.

\subsection{Numerical Evaluation of Optimized Sketch Sizes} On Figures \ref{figureoptimizedcostsrht} and \ref{figureoptimizedcostgaussian}, we show numerically that our theoretical prescriptions predict fairly well the empirical behavior. Our prescribed sketch size provides significant speed-ups over the standard choice $m=\mathcal{O}(m_\delta)$ in spite of using the rough proxy $\mathcal{C}(m)$ for the empirical computational time which does not account for hardware specifications and potentially complex interactions between the different computational steps of PCG.
\begin{figure}[h!]
	\centering
	\includegraphics[width=\textwidth]{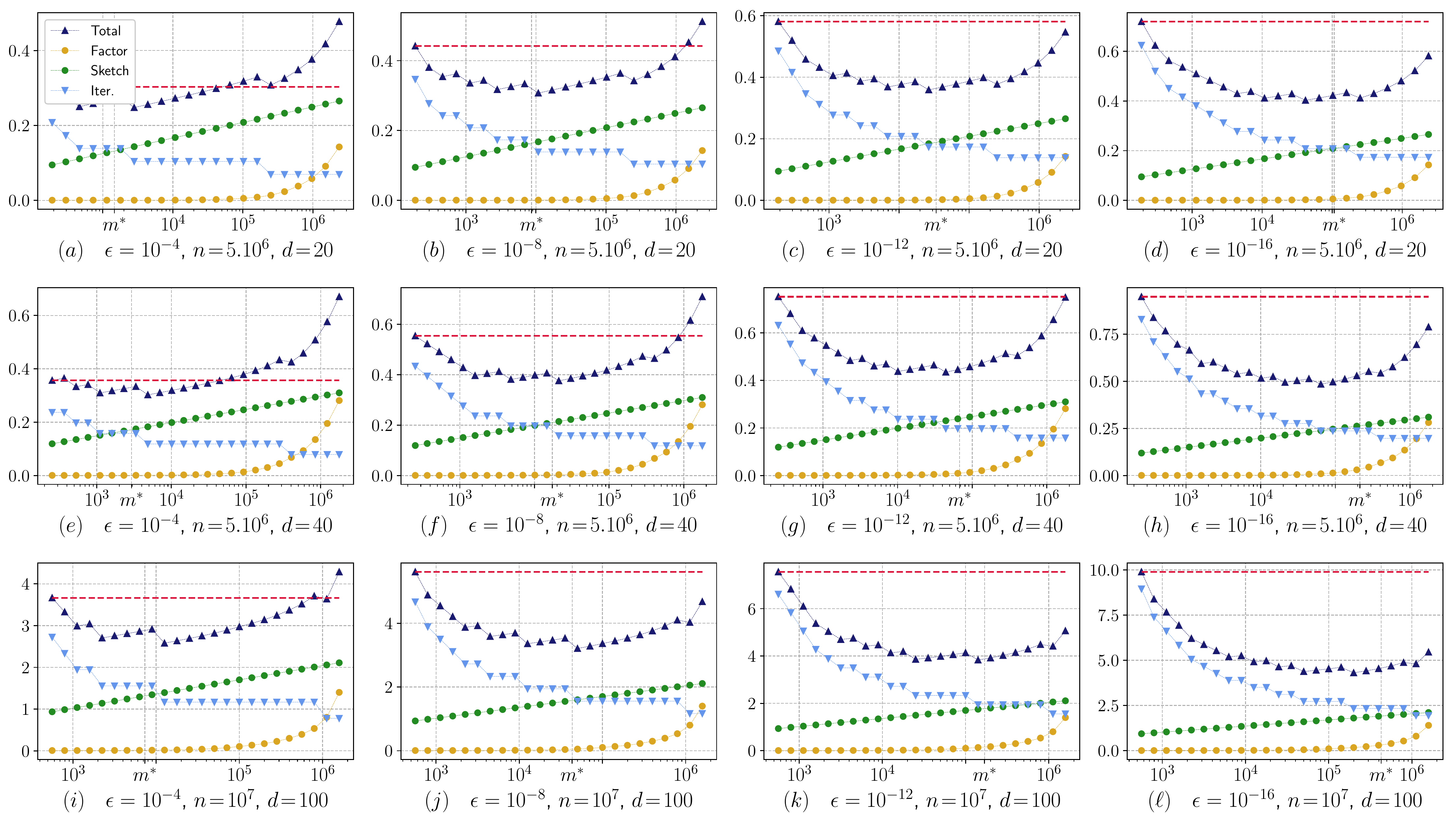}
	\caption{Sketching, factor and total iterations times (in seconds) of PCG with the SRHT to reach a $(\varepsilon, \mathcal{O}(1/d)$)-accurate solution, versus sketch size $m$. We use varying precisions $\varepsilon \in \{10^{-4}, 10^{-8}, 10^{-12}, 10^{-16}\}$ and varying dimensions $n$ and $d$ reported on each plot. Our prescribed sketch size $m^*$ is shown on the x-axis. The (red) dashed line shows the computational time of PCG with the classical prescription $m=4 d \log d$.}
	\label{figureoptimizedcostsrht}
\end{figure}
\begin{figure}[h!]
	\centering
	\includegraphics[width=\textwidth]{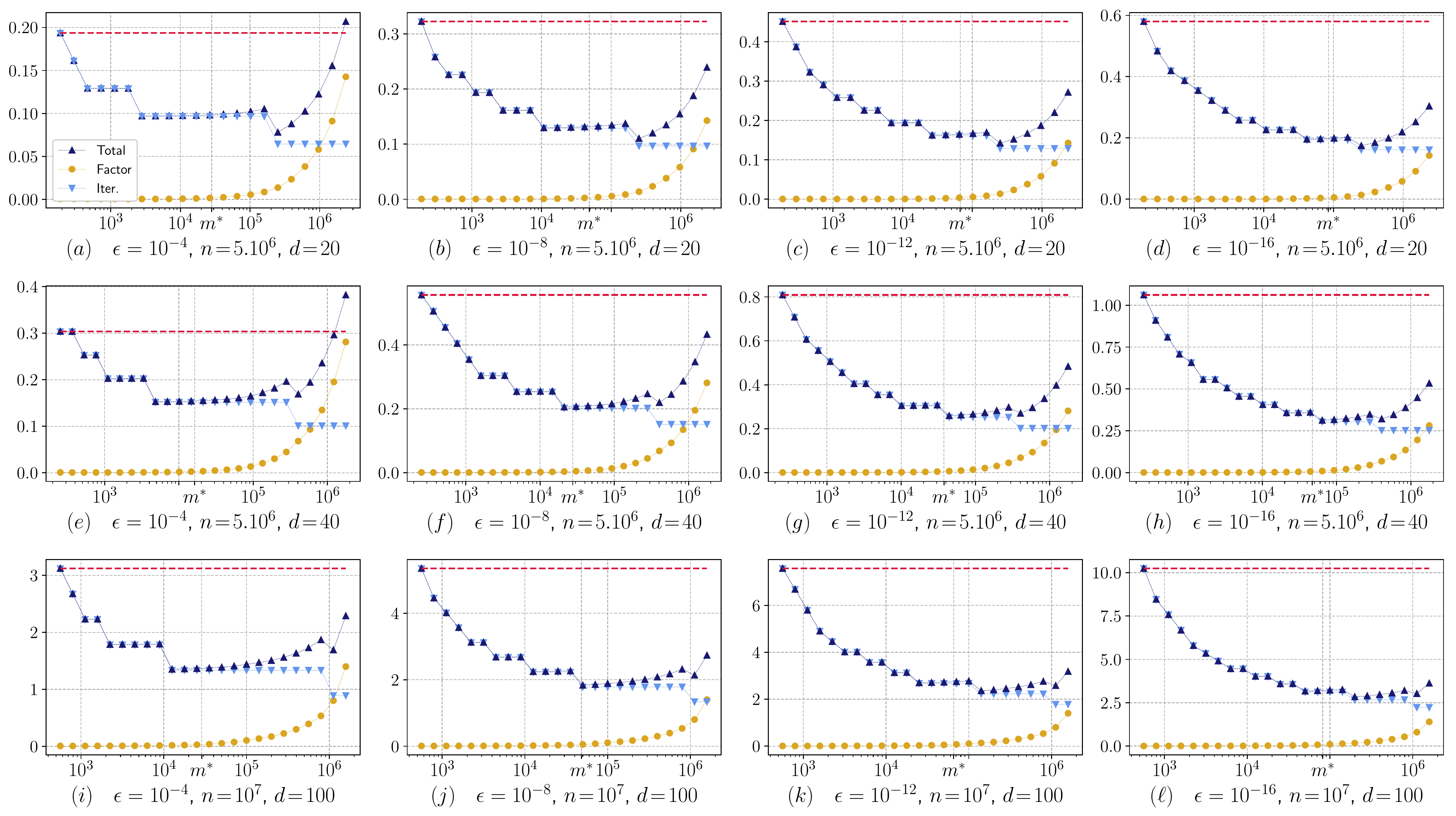}
	\caption{Factor and total iterations times (in seconds) of PCG a with Gaussian embedding to reach a $(\varepsilon, e^{-\mathcal{O}(d)})$-accurate solution, versus sketch size $m$. We use varying precisions $\varepsilon \in \{10^{-4}, 10^{-8}, 10^{-12}, 10^{-16}\}$ and varying dimensions $n$ and $d$ reported on each plot. Our prescribed sketch size $m^*$ is shown on the x-axis. The (red) dashed line shows the computational time of PCG with the classical prescription $m=4d$.}
	\label{figureoptimizedcostgaussian}
\end{figure}

\section{Proof of Theorem~\ref{theoremminimalrootradius}}
\label{sectionminimalrootradius}

We fix several notations. Recall that we defined $\eta = (1+\beta)^2 - 2\mu \theta_1 (1+\beta) + \mu^2 \theta_2$ and $\gamma = 1 + \beta - \mu \theta_1$. We denote $\eta_0 = 1 - 2\mu \theta_1 + \mu^2 \theta_2$ and $\gamma_0 = 1-\mu \theta_1$, so that $\eta = \eta_0 + 2 \gamma_0 \beta + \beta^2$ and $\gamma = \gamma_0 + \beta$. We introduce the coefficients $a_0 = -\beta^3$, $a_1 = \beta \left(\beta^2 - (1-2\gamma_0)\beta + 2 \gamma_0^2 - \eta_0 \right)$ and $a_2 = -\beta^2 + (1-2 \gamma_0)\beta - \eta_0$. A straightforward calculation yields that $\chi_{\mu,\beta}(\lambda) = \lambda^3 + a_2 \lambda^2 + a_1 \lambda + a_0$. We will also work with the variable $\alpha = \frac{\theta_2}{\theta_1} \mu$ and the reparameterized characteristic polynomial $\chi_{\alpha,\beta}(\lambda)$. We fix the notations $\mu^* = \theta_1/\theta_2$, $\alpha^* = 1$, $\beta^* = 0$ and $\rho^* = 1-\frac{\theta_1^2}{\theta_2}$

We introduce several intermediate results before developing the proof.
\begin{lemma}
\label{lemmaredundancy}
Suppose that for some parameters $\widehat{\beta} > 0$ and $\widehat{\alpha} \geq 0$, the root radius $\mathrm{\Lambda}(\widehat{\alpha}, \widehat{\beta})$ satisfies $\mathrm{\Lambda}(\widehat{\alpha}, \widehat{\beta}) < \rho^*$. Then, there exist parameters $\beta^\prime \in (0, \rho^*)$ and $\alpha^\prime \geq 0$, such that $\chi_{\alpha^\prime,\beta^\prime}(\rho^*) = 0$.
\end{lemma}
We will need the following identity; we have for any $\alpha, \beta \gre 0$ that 
\begin{align}
\label{eqnpolynomialinbeta}
    \frac{\chi_{\alpha,\beta}(\rho^*)}{(1-\rho^*)} = P_\alpha(\beta)\,,
\end{align}
where $P_\alpha(\beta) = -\beta^3 + \rho^* (1-2 \alpha) \beta^2 + \rho^* (1-\alpha) (1+\alpha (2{\rho^*} - 1)) \beta - {\rho^*}^2 (1-\alpha)^2$.
\begin{lemma}
\label{lemmanorootwithin}
Suppose that $\rho^* \neq \frac{1}{2}$. Then, for any $\alpha \geq 0$, the polynomial $P_\alpha$ has no root within the interval $(0,\rho^*)$.
\end{lemma}
We are now ready to prove Theorem~\ref{theoremminimalrootradius}. We assume by contradiction that there exist parameters $\widehat{\beta} \geq 0$ and $\widehat{\alpha} \geq 0$ such that 
\begin{align}
\label{eqncontradictionassumption}
    \Lambda(\widehat{\alpha}, \widehat{\beta}) < \rho^*\,.
\end{align}
We must have $\widehat \beta > 0$; indeed, we have shown in Theorem~\ref{theoremexacterrorrefreshed} that $\inf_{\alpha \geq 0} \Lambda(\alpha, 0) = \rho^*$.

Suppose first that $\rho^* \neq \frac{1}{2}$. From Lemma~\ref{lemmaredundancy}, we know that there must exist some parameters $0 < \beta^\prime < \rho^*$ and $\alpha^\prime \geq 0$ such that $\chi_{\alpha^\prime, \beta^\prime}(\rho^*) = 0$. According to~\eqref{eqnpolynomialinbeta}, the latter implies that the polynomial $P_{\alpha^\prime}$ has a root at $\beta^\prime \in (0, \rho^*)$. But this is contradiction with Lemma~\ref{lemmanorootwithin}.

We turn to the case $\rho^*=\frac{1}{2}$. For any $\rho^* \in (0,1)$, we denote $\mathrm{\Lambda}(\alpha,\beta,\rho^*) \equiv \mathrm{\Lambda}(\alpha,\beta)$. Note that, as $\alpha \to +\infty$ or $\beta \to +\infty$, we have $\mathrm{\Lambda}(\alpha,\beta,\rho^*) \to +\infty$. Therefore, we can restrict the range of $(\alpha, \beta)$ to a rectangle $[0,R]^2$ such that for any $\rho^* \in (0,1)$, we have
\begin{align}
    \inf_{\alpha, \beta \gre 0} \mathrm{\Lambda}(\alpha, \beta, \rho^*) = \inf_{0 \less \alpha, \beta \less R} \mathrm{\Lambda}(\alpha, \beta, \rho^*)\,.
\end{align}
We introduce the functions 
\begin{align}
f(\alpha, \beta, \rho^*) = -\rho^* + \mathrm{\Lambda}(\alpha,\beta,\rho^*)\,,\quad \mbox{and}\quad F_R(\rho^*) = \inf_{0 \less \alpha, \beta \less R} f(\alpha, \beta, \rho^*)
\end{align}
The function $f$ is continuous. Partial minimization of a continuous function over a compact preserves continuity with respect to the other variables. Therefore, $F_R$ is continuous in $\rho^*$. Further, we know that for $\rho^*\neq 1/2$, $F(\rho^*) = 0$. Hence, $F = 0$ everywhere, which implies that $\inf_{0 \less \alpha, \beta \less R} \mathrm{\Lambda}\left(\alpha,\beta,\frac{1}{2}\right) = \frac{1}{2}$. This concludes the proof of Theorem~\ref{theoremminimalrootradius}. The rest of Section~\ref{sectionminimalrootradius} aims to prove Lemma~\ref{lemmaredundancy} and Lemma~\ref{lemmanorootwithin}

\subsection{Proof of Lemma~\ref{lemmaredundancy}}

We claim that $\widehat \beta < \rho^*$. This follows from the fact that for any $\alpha, \beta \gre 0$, we have $\Lambda(\alpha,\beta) \gre \beta$. Indeed, we have $\lambda_1 \lambda_2 \lambda_3 = \beta^3$, where $\lambda_1, \lambda_2, \lambda_3$ are the roots of $\chi_{\alpha,\beta}$. Thus, $|\lambda_1 \lambda_2 \lambda_3| = \beta^3$, which further implies that $\max(|\lambda_1|, |\lambda_2|, |\lambda_3|) \gre \beta$, i.e., $\mathrm{\Lambda}(\alpha,\beta) \geq \beta$. In particular, we have $\widehat{\beta} \leq \mathrm{\Lambda}(\widehat{\alpha}, \widehat{\beta})$. Along with the assumption~\eqref{eqncontradictionassumption}, it follows that 
\begin{align}
\label{EqnBoundHatBeta}
    \widehat{\beta} < \rho^*\,.
\end{align}
We know that $\Lambda(1,0) = \rho^*$. Let $(\alpha(t), \beta(t))_{t \in [0,1]}$ be a continuous, injective path in the rectangle $[1, \widehat{\alpha}] \times [0, \widehat{\beta}]$ such that $(\alpha(0), \beta(0)) = (1,0)$, $(\alpha(1), \beta(1)) = (\widehat{\alpha}, \widehat{\beta})$ and $\beta(t) > 0$ for $t > 0$. Using~\eqref{EqnBoundHatBeta}, we have  
\begin{align*}
\beta(t)\in (0,\rho^*)\,,\quad \text{for  } t \in (0,1]\,.
\end{align*}
Denote $\Lambda(t) = \Lambda(\alpha(t), \beta(t))$. We introduce continuous parameterizations of the roots $\lambda_1(t), \lambda_2(t), \lambda_3(t)$ of $\chi_{\alpha(t), \beta(t)}$. Then, it suffices to show that one of the roots $\lambda_1(t)$, $\lambda_2(t)$ or $\lambda_3(t)$ takes the value $\rho^*$ for some $t > 0$. Indeed, by setting $\alpha^\prime = \alpha(t)$ and $\beta^\prime = \beta(t)$, it will imply the claim, since $\beta(t) \in (0, \rho^*)$.

We have that $\Lambda(0) = \rho^*$, and, $\Lambda(1) = \Lambda(\what{\alpha}, \what{\beta}) < \rho^*$. By continuity of $\Lambda(t)$ and using the fact that $\Lambda(t)$ has a strict local minimum at $t=0$ (see Lemma~\ref{lemmalocaloptimality}), there must exist $t_0 > 0$ such that $\Lambda(t_0) = \rho^*$ and $\Lambda(t) > \rho^*$ for $t \in (0,t_0)$.

Without loss of generality, we choose an indexing of the roots such that $\lambda_1(0) = \rho^*$ and $\lambda_2(0) = \lambda_3(0) = 0$. For $t$ close to $0$, by continuity, the root $\lambda_1(t)$ is not the conjugate of $\lambda_2(t)$ and $\lambda_3(t)$. Therefore, for $t$ close to $0$, $\lambda_1(t)$ must be real, equal to $\mathrm{\Lambda}(t)$ and thus, strictly greater than $\rho^*$. Since $|\lambda_1(t_0)| \less \rho^*$, there must exist $t_1 \in (0, t_0]$ such that $|\lambda_1(t_1)| = \rho^*$. Either $\lambda_1(t_1) = \rho^*$, which concludes the proof. Or, $\lambda_1(t_1)$ is strictly complex or equal to $-\rho^*$. In both cases, by continuity, there must exist $t \in (0, t_0)$ such that $\lambda_1(t)$ is strictly complex. Denote by $t_2$ the infimum time at which $\lambda_1(t)$ becomes strictly complex. It holds that $t_2 > 0$, since $\lambda_1(0)$ has single multiplicity. For $t \in (0, t_2]$, the root $\lambda_1(t)$ is real. Either there exists $t \in (0,t_2]$ such that $\lambda_1(t) = \rho^*$, which concludes the proof. Or, $\lambda_1(t) > \rho^*$ for all $t \in (0, t_2]$. By conjugacy of the complex roots, we must have $\lambda_1(t_2) = \lambda_i(t_2) > \rho^*$ for some $i\in\{2,3\}$ (without loss of generality, say $i=2$). Since $\lambda_2(0)=0$, by continuity of the root, there must exist $t_3 \in (0,t_2)$ such that $|\lambda_2(t_3)| = \rho^*$.

Either the root $\lambda_2(t)$ crosses the (complex) circle of radius $\rho^*$ along the real axis, at the point $\rho^*$, which concludes the proof.

Or, the root $\lambda_2(t)$ crosses the circle of radius $\rho^*$ in the (strictly) complex plane or at $-\rho^*$, and then it hits the real axis $[\rho^*,+\infty)$. Denote $t_4$ the first time at which $\lambda_2(t)$ hits the real axis $[\rho^*, +\infty)$. Then, by conjugacy of $\lambda_2(t_4)$ and $\lambda_3(t_4)$ (since, right before $t_4$, $\lambda_1(t)$ is real and $\lambda_2(t)$ must be complex and hence, conjugate to $\lambda_3(t)$), we must have $\lambda_2(t_4) = \lambda_3(t_4)$. Hence, $\lambda_1(t_4) \lambda_2(t_4) \lambda_3(t_4) > {\rho^*}^3$, which yields that $\beta(t_4)^3 < {\rho^*}^3 < \lambda_1(t_4) \lambda_2(t_4) \lambda_3(t_4) = \beta(t_4)^3$. The latter set of inequalities yields a contradiction, and thus the claim.

\subsubsection{Intermediate Results for the Proof of Lemma~\ref{lemmaredundancy}}

\begin{lemma}[Local optimality]
\label{lemmalocaloptimality}
The root radius function $(\mu, \beta) \mapsto \Lambda(\mu, \beta)$ has a strict local minimum at $(\mu^*, \beta^*) = (\theta_2^{-1}\theta_1, 0)$ over $\real^2$.
\end{lemma}
\begin{proof}
It is straightforward to check that the roots of $\chi_{\mu^*,\beta^*}$ are equal to $\rho^* \defn 1-\frac{\theta_1^2}{\theta_2}$ with single multiplicity and $0$ with double multiplicity. We denote by $\lambda_1 \equiv \lambda_1(\mu, \beta)$, $\lambda_2 \equiv \lambda_2(\mu, \beta)$ and $\lambda_3 \equiv \lambda_3(\mu, \beta)$ continuous parameterizations of the roots of $\chi_{\mu,\beta}$. Without loss of generality, we assume that $\lambda_1(\mu^*,\beta^*) = \rho^*$, $\lambda_2(\mu^*,\beta^*) = 0$, and, $\lambda_3(\mu^*,\beta^*) = 0$.

For $(\mu,\beta)$ in a neighborhood of $(\mu^*,\beta^*)$, by continuity of $(\lambda_1, \lambda_2, \lambda_3)$, we have that $\lambda_2$ and $\lambda_3$ are close to $0$ and $\lambda_1$ close to $\rho^*$, whence $\lambda_1$ is necessarily a real root of the polynomial $\chi_{\beta, \mu}$ with single multiplicity (otherwise, $\lambda_1$ would be the complex conjugate of $\lambda_2$ or $\lambda_3$, but $\lambda_1$ is far apart from $\lambda_2$ and $\lambda_3$). Since $\lambda_1$ is real and has single multiplicity, the function $\lambda_1$ is differentiable in a neighborhood of $(\mu^*, \beta^*)$.

Recall that we denote by $a_i$ the $i$-th order coefficient of the three-degree polynomial $\chi_{\mu,\beta}(\lambda)$ and that $a_3=1$. The function $\beta \mapsto \chi_{\mu, \beta}(\lambda_1(\mu, \beta))$ is constant and equal to $0$. Hence, we have for $(\mu, \beta)$ close to $(\mu^*,\beta^*)$ that
\begin{align}
    0 = \frac{\mathrm{d}\chi_{\mu, \beta}(\lambda_1)}{\mathrm{d}\beta} = \frac{\mathrm{d}\lambda_1}{\mathrm{d}\beta} \left(3 \lambda_1^2 + 2\lambda_1 a_2 + a_1\right) + \lambda_1^2 \frac{\mathrm{d}a_2}{\mathrm{d}\beta} + \lambda_1 \frac{\mathrm{d}a_1}{\mathrm{d}\beta} + \frac{\mathrm{d}a_0}{\mathrm{d}\beta} \label{EqnMainEqualityLocalOptimality}\,.
\end{align}
At $(\mu, \beta)=(\mu^*,\beta^*)$, we have $a_2 = -\eta_0$, $a_1 = 0$, $a_0 = 0$, $\frac{\mathrm{d}a_2}{\mathrm{d}\beta} = 1 - 2 \gamma_0$, $\frac{\mathrm{d}a_1}{\mathrm{d}\beta} = 2\gamma_0^2 - \eta_0$, $\frac{\mathrm{d}a_0}{\mathrm{d}\beta} = 0$, $\gamma_0 = \rho^*$ and $\eta_0 = \rho^*$. In particular, we have $3 {\rho^*}^2 + 2\rho^* a_2 + a_1 = {\rho^*}^2 \neq 0$. By continuity over a neighborhood of $(\mu^*,\beta^*)$, the term $3 \lambda_1^2 + 2\lambda_1 a_2 + a_1$ is non-zero, and, from~\eqref{EqnMainEqualityLocalOptimality}, we obtain $\frac{\mathrm{d}\lambda_1}{\mathrm{d}\beta} \!=\! - \frac{\lambda_1^2 \frac{\mathrm{d}a_2}{\mathrm{d}\beta} + \lambda_1 \frac{\mathrm{d}a_1}{\mathrm{d}\beta} + \frac{\mathrm{d}a_0}{\mathrm{d}\beta}}{3 \lambda_1^2 + 2\lambda_1 a_2 + a_1}$. In particular, it implies that $\lambda_1$ is infinitely differentiable with respect to $\beta$, around $\beta^*$. Evaluating the latter derivative at $(\mu^*,\beta^*)$, we find that $\frac{\mathrm{d}\lambda_1}{\mathrm{d}\beta}(\mu^*, \beta^*) \!=\! 0$.  

Using similar arguments, we obtain that around $(\mu^*,\beta^*)$,
\begin{align}
0 = \frac{\mathrm{d}\chi_{\mu,\beta}(\lambda_1)}{\mathrm{d}\mu}
= \frac{\mathrm{d}\lambda_1}{\mathrm{d}\mu} \left(3 \lambda_1^2 + 2\lambda_1 a_2 + a_1\right) + \lambda_1^2 \frac{\mathrm{d}a_2}{\mathrm{d}\mu} + \lambda_1 \frac{\mathrm{d}a_1}{\mathrm{d}\mu} + \frac{\mathrm{d}a_0}{\mathrm{d}\mu}\label{EqnMainEquality2LocalOptimality}\,,
\end{align}
whence $\frac{\mathrm{d}\lambda_1}{\mathrm{d}\mu} = - \frac{\lambda_1^2 \frac{\mathrm{d}a_2}{\mathrm{d}\mu} + \lambda_1 \frac{\mathrm{d}a_1}{\mathrm{d}\mu} + \frac{\mathrm{d}a_0}{\mathrm{d}\mu}}{3 \lambda_1^2 + 2\lambda_1 a_2 + a_1}$. At $(\mu^*,\beta^*)$, we have $\frac{\mathrm{d}a_2}{\mathrm{d}\mu} \!=\! \frac{\mathrm{d}a_1}{\mathrm{d}\mu} \!=\! \frac{\mathrm{d}a_0}{\mathrm{d}\mu} \!=\! 0$, and thus, $\frac{\mathrm{d}\lambda_1}{d\mu}(\mu^*,\beta^*) \!=\! 0$.

Differentiating again~\eqref{EqnMainEqualityLocalOptimality} with respect to $\beta$, evaluating at $(\mu^*,\beta^*)$ and using the fact that $\frac{\mathrm{d}\lambda_1}{\mathrm{d}\beta}(\mu^*,\beta^*) \!=\! 0$, we get
\begin{align*}
0 = \frac{\mathrm{d}^2\lambda_1}{\mathrm{d}\beta^2} \left(3 {\rho^*}^2 + 2\rho^* a_2 + a_1\right) + {\rho^*}^2 \frac{\mathrm{d}^2a_2}{\mathrm{d}\beta^2} + \rho^* \frac{\mathrm{d}^2a_1}{\mathrm{d}\beta^2} + \frac{\mathrm{d}^2a_0}{\mathrm{d}\beta^2}\,,
\end{align*}
and consequently, $\frac{\mathrm{d}^2\lambda_1}{\mathrm{d}\beta^2} = -\frac{{\rho^*}^2 \frac{\mathrm{d}^2a_2}{\mathrm{d}\beta^2} + \rho^* \frac{\mathrm{d}^2 a_1}{\mathrm{d}\beta^2} + \frac{\mathrm{d}^2a_0}{\mathrm{d}\beta^2}}{3 {\rho^*}^2 + 2\rho^* a_2 + a_1}$. At $(\mu^*, \beta^*)$, we find $\frac{\mathrm{d}^2a_0}{\mathrm{d}\beta^2} = 0$, $\frac{\mathrm{d}^2a_1}{\mathrm{d}\beta^2} = 4 \rho^* - 2$ and $\frac{\mathrm{d}^2a_2}{\mathrm{d}\beta^2} = -2$. Hence, we obtain $\frac{\mathrm{d}^2\lambda_1}{\mathrm{d}\beta^2}(\mu^*,\beta^*) = 2\,\frac{1-\rho^*}{\rho^*}$. Differentiating~\eqref{EqnMainEqualityLocalOptimality} with respect to $\mu$, evaluating at $(\mu^*,\beta^*)$ and using the fact that $\frac{\mathrm{d}\lambda_1}{\mathrm{d}\mu}(\mu^*,\beta^*) \!=\! 0$, we get 
\begin{align*}
0 = \frac{\mathrm{d}^2\lambda_1}{\mathrm{d}\beta \mathrm{d}\mu} \left(3 {\rho^*}^2 + 2\rho^* a_2 + a_1\right) + {\rho^*}^2 \frac{\mathrm{d}^2a_2}{\mathrm{d}\beta \mathrm{d}\mu} + \rho^* \frac{\mathrm{d}^2a_1}{\mathrm{d}\beta \mathrm{d}\mu} + \frac{\mathrm{d}^2 a_0}{\mathrm{d}\beta \mathrm{d}\mu}\,,
\end{align*}
At $(\mu^*,\beta^*)$, we have $\frac{\mathrm{d}^2a_2}{\mathrm{d}\beta \mathrm{d}\mu} = 2\theta_1$, $\frac{\mathrm{d}^2a_1}{\mathrm{d}\beta \mathrm{d}\mu} = -4\theta_1 \rho^*$ and $\frac{\mathrm{d}^2a_0}{\mathrm{d}\beta \mathrm{d}\mu} = 0$, and consequently, $\frac{\mathrm{d}^2\lambda_1}{\mathrm{d}\beta \mathrm{d}\mu}(\mu^*,\beta^*) = 2 \theta_1$. 

Differentiating~\eqref{EqnMainEquality2LocalOptimality} with respect to $\mu$, evaluating at $(\mu^*,\beta^*)$ and using the fact that $\frac{\mathrm{d}\lambda_1}{d\mu}(\mu^*,\beta^*) = 0$, we get
\begin{align*}
0 = \frac{\mathrm{d}^2\lambda_1}{\mathrm{d}\mu^2} \left(3 {\rho^*}^2 + 2\rho^* a_2 + a_1\right) + {\rho^*}^2 \frac{\mathrm{d}^2a_2}{\mathrm{d}\mu^2} + \rho^* \frac{\mathrm{d}^2a_1}{\mathrm{d}\mu^2} + \frac{\mathrm{d}^2a_0}{\mathrm{d}\mu^2}\,.
\end{align*}
At $(\mu^*,\beta^*)$, we find $\frac{\mathrm{d}^2a_0}{\mathrm{d}\mu^2} = 0$, $\frac{\mathrm{d}^2a_1}{\mathrm{d}\mu^2} = 0$ and $\frac{\mathrm{d}^2a_2}{\mathrm{d}\mu^2} = -2\theta_2$, whence $\frac{\mathrm{d}^2\lambda_1}{\mathrm{d}\mu^2}(\mu^*,\beta^*) = 2 \theta_2$. 

Collecting these second-order derivatives, we find that the Hessian of $\lambda_1$ at $(\mu^*,\beta^*)$ is 
\begin{align}
    M_{\lambda_1} = 2\, \begin{bmatrix} \frac{\theta_1^2/\theta_2}{1-\theta_1^2/\theta_2} & \theta_1 \\ \theta_1 & \theta_2\end{bmatrix}\,.
\end{align}
We set $x \defn \frac{\theta_1^2}{\theta_2}$. Since $x \in (0,1)$ (this follows from Cauchy-Schwartz inequality), we find that $\mbox{tr}\,M_{\lambda_1} = 2 \frac{x}{1-x} + 2 \theta_2$, which is positive. Further, $\mbox{det}\,M_{\lambda_1} = 4\theta_2 \frac{x^2}{1-x}$ is also positive. Therefore, the Hessian $M_{\lambda_1}$ is positive definite and $(\mu^*,\beta^*)$ is a strict local minimum. 
\end{proof}

\subsection{Proof of Lemma~\ref{lemmanorootwithin}}
\label{prooflemmanorootwithin}
We fix some notations. We recall (see~\eqref{eqnpolynomialinbeta}) that $P_{\alpha}(\beta) = a \, \beta^3 + b \, \beta^2 + c \, \beta + d$, where $a = -1$, $b \equiv b(\alpha,\rho^*) = \rho^* (1-2\alpha)$, $c \equiv c(\alpha, \rho^*) = \rho^* (1-\alpha) (1+\alpha (2\rho^* - 1))$ and $d \equiv d(\alpha,\rho^*) = -{\rho^*}^2 (1-\alpha)^2$. In this proof, we reduce the complexity of studying the degree three polynomial $P_\alpha$ by an analysis of the roots of its degree two derivative $P^\prime_\alpha$. Our proof involves the maximal root $\beta_+(\alpha)$ of $P^\prime_\alpha$ that we formally define in the next result.
\begin{lemma}
\label{lemmamaximalrealroot}
The degree two polynomial $P_{\alpha}^\prime$ has two distinct real roots, and its maximal root $\beta_+(\alpha)$ is given by $\beta_+(\alpha) \equiv \beta_+(\alpha, \rho^*) = \frac{1}{3} \left[ b + \sqrt{b^2 + 3 c} \right]$. Furthermore, we have that $P_{\alpha}(\beta_+(\alpha)) = \frac{1}{27} \big[27 d + 9bc + 2 b^3 + 2\!\left(b^2 + 3c\right)^{\frac{3}{2}}\big]$.
\end{lemma}
\begin{proof}
A straightforward calculation yields that the discriminant of $P_\alpha^\prime$ is equal to $b^2 + 3c$, and we have $b^2 + 3c = \frac{\rho^*}{3-2 \rho^*} \big(((3-2\rho^*)\alpha + \rho^*-3)^2 + 3\rho^*(1-\rho^*)\big)$. We see that $b^2 + 3c > 0$, whence $P_\alpha^\prime$ has two distinct real roots. A straightforward calculation yields the claimed formula for $\beta_+(\alpha)$ and $P_\alpha(\beta_+(\alpha))$.
\end{proof}
We proceed by contradiction to prove Lemma~\ref{lemmanorootwithin}, that is, we assume that $P_\alpha$ has a real root within the interval $(0,{\rho^*})$. In Lemma~\ref{lemmadoubleroot}, we show that this implies the existence of another parameter $\widehat \alpha > 0$ such that  $\widehat \alpha \neq 1$, $\beta(\widehat{\alpha})$ is a root of $P_{\widehat{\alpha}}$ and $\beta(\widehat{\alpha}) = \beta_+(\widehat{\alpha}) \in (0, {\rho^*})$. On the other hand, we show in Lemma~\ref{lemmalocalmaximum} that for any $\alpha \gre 0$ such that $\beta_+(\alpha) \in (0,\rho^*)$, we must have $P_\alpha(\beta_+(\alpha)) < 0$. Applying Lemma~\ref{lemmalocalmaximum} to $\alpha = \widehat \alpha$ yields a contradiction and concludes the proof of Lemma~\ref{lemmanorootwithin}. It remains to state and prove the intermediate results Lemma~\ref{lemmadoubleroot} and Lemma~\ref{lemmalocalmaximum}.

\subsubsection{Intermediate Results for the Proof of Lemma~\ref{lemmanorootwithin}}

\begin{lemma}
\label{lemmadoubleroot}
Suppose that $\rho^* \neq \frac{1}{2}$ and that, for some $\overline{\alpha} \geq 0$, the polynomial $P_{\overline{\alpha}}$ has a real root within the interval $(0,{\rho^*})$. Then, there must exist some $\widehat{\alpha} > 0$ such that $\widehat{\alpha} \neq 1$, $\beta_+(\widehat{\alpha}) = \beta(\widehat{\alpha}) \in (0, {\rho^*})$ and $P_{\widehat{\alpha}}(\beta(\widehat{\alpha})) = 0$.
\end{lemma}
\begin{proof}
According to Lemma~\ref{lemmarootwithin}, we must have $\overline{\alpha} > 0$ and $\overline{\alpha} \neq 1$. Further, $P_{\overline{\alpha}}$ must either have one real root $\beta$ with multiplicity two within $(0,{\rho^*})$ and $\beta = \beta_+(\overline{\alpha})$ (in which case the result holds by setting $\widehat{\alpha} = \overline{\alpha}$), or, $P_{\overline{\alpha}}$ must have two distinct real roots $\beta_1, \beta_2 \in (0,{\rho^*})$ such that 
\begin{align}
\label{eqnsandwichinequality}
    \beta_1 < \beta_+(\overline{\alpha}) < \beta_2\,.
\end{align}
Thus, we assume that~\eqref{eqnsandwichinequality} holds. We introduce $\alpha \mapsto \beta_1(\alpha)$ and $\alpha \mapsto \beta_2(\alpha)$ continuous parameterizations over $[0, +\infty)$ of two of the roots of the polynomial $P_\alpha$, such that $\beta_1(\overline{\alpha}) = \beta_1$ and $\beta_2(\overline{\alpha}) = \beta_2$. Let us distinguish two cases.

\textbf{Case 1: $\overline{\alpha} < 1$.} We assume, by contradiction, that for any $\alpha \in [0,\overline{\alpha}]$, $\beta_1(\alpha) \neq \beta_2(\alpha)$. For any $\alpha \in [0,\overline{\alpha}]$, $P_\alpha(0) = -{\rho^*}^2 (1-\alpha)^2 \neq 0$. Therefore, as $\alpha$ decreases from $\overline{\alpha}$ to $0$, the roots $\beta_1(\alpha)$ and $\beta_2(\alpha)$ cannot become equal to $0$, nor can its third root which must be negative according to Lemma~\ref{lemmanegativeroot}. Either $\beta_1(\alpha)$ and $\beta_2(\alpha)$ remain both real and, thus, positive. Then, we must have $\beta_1(\alpha) < \beta_2(\alpha)$ for any $\alpha \in [0,\overline{\alpha}]$, and the third root of $P_\alpha$ remains negative, whence $\beta_2(\alpha)$ is the maximal root of $P_\alpha$ for any $\alpha \in [0,\overline{\alpha}]$. But, it holds that $\beta_2(\overline{\alpha}) < {\rho^*} < \sqrt{{\rho^*}}$, and $P_0$ has roots $\pm \sqrt{{\rho^*}}$ and ${\rho^*}$. It implies that $\beta_2(0) = \sqrt{{\rho^*}}$, so that, by continuity, $\beta_2(\alpha)$ must be equal to ${\rho^*}$ for some $\alpha \in (0,\overline{\alpha})$. But $P_\alpha({\rho^*}) = 0$ implies that $\alpha = 0$, which is a contradiction. Hence, $\beta_1(\alpha)$ or $\beta_2(\alpha)$ (say $\beta_1$) must become strictly complex for some $\alpha \in (0, \overline{\alpha})$. Let $\widehat{\alpha} \defn \sup \{\alpha < \overline{\alpha} \mid \beta_1(\alpha) \in \mathbb{C} - \mathbb{R}\}$. By continuity of $\beta_1(\alpha)$ and the fact that $\beta_1(\overline{\alpha}) \in \mathbb{R}$, we must have $\beta_1(\widehat{\alpha}) \in \mathbb{R}$. Since $\beta_1(\alpha)$ cannot cross the point $0$ nor the point ${\rho^*}$ for $\alpha \in (\widehat{\alpha}, \overline{\alpha})$, it holds that $\beta_1(\widehat{\alpha}) \in (0, {\rho^*})$. By continuity and conjugacy, we must have that $\beta_1(\widehat{\alpha}) = \beta_2(\widehat{\alpha})$, which is a contradiction.

\textbf{Case 2: $\overline{\alpha} > 1$.} We assume by contradiction that for any $\alpha \geq \overline{\alpha}$, $\beta_1(\alpha) \neq \beta_2(\alpha)$. Either $\beta_1(\alpha)$ and $\beta_2(\alpha)$ remain both real, and we must have $\beta_1(\alpha) < \beta_2(\alpha)$ for any $\alpha \geq \overline{\alpha}$. If ${\rho^*} < \frac{1}{2}$, then we have $\beta_+(\alpha) \to +\infty$ as $\alpha \to +\infty$ (see Lemma~\ref{lemmaasymptoticsbetaplus}). Combined with~\eqref{eqnsandwichinequality}, this implies that $\beta_2(\alpha) \to +\infty$ and $\beta_2(\alpha)$ must be equal to ${\rho^*}$ for some $\alpha > \overline{\alpha}$. But $P_\alpha(\rho^*)=0$ implies that $\alpha=0$, which is a contradiction. If ${\rho^*} > \frac{1}{2}$, then $\beta_+(\alpha) \to -\infty$ as $\alpha \to + \infty$, so that $\beta_1(\alpha) \to - \infty$ and $\beta_1(\alpha)$ must be equal to $0$ for some $\alpha > \overline{\alpha}$. But $P_\alpha(0) = -{\rho^*}^2 (1-\alpha)^2$, which cannot be equal to $0$ for $\alpha \gre \overline{\alpha} > 1$, and we have a contradiction. Hence, $\beta_1(\alpha)$ or $\beta_2(\alpha)$ (say $\beta_1$) must become strictly complex for some $\alpha \gre \overline{\alpha}$. Let $\widehat{\alpha} \defn \inf \{\alpha > \overline{\alpha} \mid \beta_1(\alpha) \in \mathbb{C} - \mathbb{R}\}$. By continuity of $\beta_1(\alpha)$ and the fact that $\beta_1(\overline{\alpha}) \in \mathbb{R}$, we must have $\beta_1(\widehat{\alpha}) \in \mathbb{R}$. Since $\beta_1(\alpha)$ cannot cross the point $0$ nor the point ${\rho^*}$ for $\alpha \in ( \overline{\alpha}, \widehat{\alpha})$, it holds that $\beta_1(\widehat{\alpha}) \in (0, {\rho^*})$. By continuity and conjugacy, we must have that $\beta_1(\widehat{\alpha}) = \beta_2(\widehat{\alpha})$, which is a contradiction.
\end{proof}

\begin{lemma}
\label{lemmalocalmaximum}
Suppose that $\rho^* \neq 1/2$. It holds that $P_\alpha(\beta_+(\alpha)) < 0$ for any $\alpha \gre 0$ such that $\beta_+(\alpha) \in (0,\rho^*)$.
\end{lemma}
%
%
\begin{proof}
Using the expression of $P_\alpha(\beta_+(\alpha))$ given in Lemma~\ref{lemmamaximalrealroot}, a simple calculation yields that $\frac{\mathrm{d}}{\mathrm{d}\alpha}P_\alpha(\beta_+(\alpha))_{|\alpha=1} = 0$ and $\frac{\mathrm{d}^2}{\mathrm{d}\alpha^2}P_\alpha(\beta_+(\alpha))_{|\alpha=1} = -2 (1-{\rho^*}) {\rho^*}^2$. Therefore, $\alpha=1$ is a strict local maximum of $\alpha \mapsto P_\alpha(\beta_+(\alpha))$, and we find through further calculation that $P_1(\beta_+(1)) = 0$.

We make the assumption that the function $P_\alpha(\beta_+(\alpha))$ has a unique local maximum at $\alpha=1$. \textbf{Case 1: ${\rho^*} < \frac{1}{2}$.} From Lemma~\ref{lemmarangealpha}, we know that there exists $\alpha_1, \alpha_2$ such that $0 < \alpha_1 < 1 < \alpha_2$, $\beta_+(\alpha_1) = \beta_+(\alpha_2) = \rho^*$ and such that for any $\alpha \gre 0$, $\beta_+(\alpha) < \rho^*$ if and only if $\alpha \in (\alpha_1, \alpha_2)$. It is easily verified that $P_{\alpha}(\rho^*) < 0$ for any $\alpha \neq 1$, whence $P_{\alpha_i}(\beta_+(\alpha_i)) < 0 = P_{1}(\beta_+(1))$ for $i=1,2$. Combining the latter assumption and the fact that $1 \in (\alpha_1, \alpha_2)$, the maximum of $\alpha \mapsto P_\alpha(\beta_+(\beta))$ over $[\alpha_1,\alpha_2]$ is uniquely attained at $\alpha=1$. Hence, for any $\alpha \in (\alpha_1,\alpha_2)$ and $\alpha \neq 1$, we get that $P_\alpha(\beta_+(\alpha)) < P_1(\beta_+(1))$, i.e., $P_\alpha(\beta_+(\alpha)) < 0$. Equivalently, for any $\alpha \gre 0$ such that $\beta_+(\alpha) \in (0,\rho^*)$, we have $P_\alpha(\beta_+(\alpha)) < 0$ which is the claimed result. \textbf{Case 2: ${\rho^*} > \frac{1}{2}$.} From Lemma~\ref{lemmarangealpha}, we know that there exists $\overline \alpha \in (0,1)$ such that $\beta_+(\overline \alpha) = \rho^*$ and such that for any $\alpha \gre 0$, $0 < \beta_+(\alpha) < \rho^*$ if and only if $\alpha \in (\overline \alpha, 1)$. Since $\overline \alpha \neq 1$, we must have $P_{\overline{\alpha}}(\rho^*) < 0$, i.e., $P_{\overline{\alpha}}(\beta_+(\overline{\alpha})) < P_1(\beta_+(1))$. By uniqueness of the local maximum, we deduce that $P_\alpha(\beta_+(\alpha)) < 0$ for any $\alpha \in (\overline \alpha, 1)$. Equivalently, for any $\alpha \gre 0$ such that $\beta_+(\alpha) \in (0,\rho^*)$, we have $P_\alpha(\beta_+(\alpha)) < 0$ which is the claimed result.

It remains to show that the above assumption holds true, i.e., $\alpha \mapsto P_{\alpha}(\beta_+(\alpha))$ has a unique local maximum at $\alpha=1$. A simple calculation yields that $P_\alpha(\beta_+(\alpha)) = Q(\alpha) + \frac{2}{27} \!\left(b^2 + 3c\right)^{\frac{3}{2}}$ where $Q(\alpha) \defn \frac{1}{27} \left(27 d + 9bc + 2 b^3 \right)$. Then, $\frac{\mathrm{d}}{\mathrm{d}\alpha}P_\alpha(\beta_+(\alpha)) = Q^\prime(\alpha) + R(\alpha)$ where $R(\alpha) \defn \frac{1}{9} (2b b^\prime + 3c^\prime) \left(b^2 + 3c\right)^{\frac{1}{2}}$. By definition, any critical point of $\alpha \mapsto P_\alpha(\beta_+(\alpha))$ is a solution of the equation $Q^\prime(\alpha) = -R(\alpha)$. Squaring both sides of the latter equation, we get that any critical point must satisfy $S(\alpha) \defn {Q^\prime(\alpha)}^2 - R(\alpha)^2 = 0$. Through further calculation, we find that the function $S(\alpha)$ is a polynomial of degree less than four. Thus, the function $\alpha \mapsto P_\alpha(\beta_+(\alpha))$ has at most four critical points. Suppose by contradiction that there exist at least two local maxima. From Lemma~\ref{lemmapbetaplusasymp}, we know that $P_{\alpha}(\beta_+(\alpha)) \to +\infty$ as $|\alpha| \to +\infty$. The latter fact along with the existence of (at least) two local maxima implies that there must exist (at least) three local minima. Thus, there exist at least five critical points, which is a contradiction, and concludes the proof.
\end{proof}

\subsubsection{Additional Helper Results}

\begin{lemma}
\label{lemmarootwithin}
Suppose that $\rho^* \neq 1/2$ and let $\alpha \gre 0$. Suppose that $P_\alpha(\beta) = 0$ for some $\beta \in (0,\rho^*)$. Then, it holds that $\alpha > 0$ and $\alpha \neq 1$. Furthermore, it holds that, either $P_\alpha$ has two distinct real roots $\beta_1(\alpha), \beta_2(\alpha) \in (0,\rho^*)$ such that $\beta_1(\alpha) < \beta_+(\alpha) < \beta_2(\alpha)$, or,  $P_\alpha$ has one real root $\beta(\alpha)$ with multiplicity two within $(0,\rho^*)$ and $\beta(\alpha) = \beta_+(\alpha)$. 
\end{lemma}
\begin{proof}
Suppose that the polynomial $P_\alpha$ has a root $\widetilde{\beta} \in (0,\rho^*)$. We easily obtain that $P_{0}(\widetilde{\beta}) = - (\widetilde{\beta}-{\rho^*})(\widetilde{\beta}-\sqrt{{\rho^*}}) (\widetilde{\beta} + \sqrt{{\rho^*}})$, which cannot be equal to $0$ since $-\sqrt{\rho^*} < 0 < \widetilde{\beta} < \rho^* < \sqrt{\rho^*}$. On the other hand, we find that $P_{1}(\beta) = -\beta^3 - {\rho^*} \beta^2$ whose roots are exactly $0$ with multiplicity two and $-{\rho^*}$. Hence, $P_{1}$ has no root within $(0,{\rho^*})$. Therefore, we must have $\alpha > 0$ and $\alpha \neq 1$. We claim that, either $P_\alpha$ has a second root within $(0,{\rho^*})$, or, its root $\widetilde{\beta}$ has multiplicity two. It is easily verified that $P_\alpha(0) < 0$ and $P_\alpha({\rho^*}) < 0$. First, assume that $P_\alpha(\beta) \leq 0$ in a neighborhood of $\widetilde{\beta}$, i.e., the root $\widetilde{\beta}$ is a local maximum. In that case, $P_\alpha^\prime(\widetilde{\beta})=0$, which implies that $\widetilde{\beta}$ is a real root of $P_\alpha$ with multiplicity at least two. Since $P_\alpha(-\infty) = +\infty$ and $P_\alpha(0) < 0$, we must have $\widetilde{\beta} = \beta_+(\alpha)$. On the other hand, if $\widetilde{\beta}$ is not a local maximum, then $P_\alpha$ takes both negative and positive values within $(0,{\rho^*})$. Since $P_\alpha(0) < 0$ and $P_\alpha({\rho^*}) < 0$, we obtain that $P_\alpha$ must cross the $x$-axis at least twice, at $\beta_1(\alpha)$ and $\beta_2(\alpha)$, and that $P_{\alpha}(\beta) > 0$ for $\beta \in (\beta_1(\alpha), \beta_2(\alpha))$, which further implies that $\beta_1(\alpha) < \beta_+(\alpha) < \beta_2(\alpha)$.
\end{proof}
\begin{lemma}
\label{lemmanegativeroot}
Suppose that $\rho^* \neq 1/2$. If $\alpha \neq 1$, then $P_{\alpha}$ has a negative root. If $\alpha = 1$, then $P_\alpha(0)=0$.
\end{lemma}
\begin{proof}
If $\alpha \neq 1$, then the zero-order coefficient of $P_\alpha$, which is equal to $-{\rho^*}^2 (1-\alpha)^2$, is negative. Since the degree of $P_\alpha$ is odd and its dominant coefficient is negative, it follows that $P_\alpha$ must have a negative root. If $\alpha = 1$, the zero-order coefficient is equal to $0$, i.e., $P_1(0) = 0$.
\end{proof}
\begin{lemma}
\label{lemmaasymptoticsbetaplus}
If $\rho^* < \frac{1}{2}$, then  $\beta_+(\alpha) \to +\infty$ as $|\alpha| \to +\infty$. If $\rho^* > \frac{1}{2}$, then $\beta_+(\alpha) \to \pm \infty$ as $\alpha \to \mp \infty$.
\end{lemma}
\begin{proof}
We have
\begin{align*}
\beta_+(\alpha) = \frac{1}{3} \left[ {\rho^*}(1-2\alpha) + \sqrt{\frac{{\rho^*}}{3-2{\rho^*}}} \sqrt{\left((3-2{\rho^*})\alpha + {\rho^*} - 3\right)^2 + 3{\rho^*}(1-{\rho^*})} \right]\,.
\end{align*}
Thus, it holds that $\beta_+(\alpha) = \frac{1}{3} \left(\sqrt{{\rho^*} (3-2{\rho^*})} \,|\alpha|  - 2{\rho^*}\, \alpha \right) + \mathcal{O}(1)$ as $|\alpha| \to +\infty$. Hence, the asymptotic limits of $\beta_+(\alpha)$ immediately follow from the fact that the inequality $\sqrt{{\rho^*} (3-2{\rho^*})} > 2 {\rho^*}$ is equivalent to ${\rho^*} < \frac{1}{2}$.
\end{proof}
\begin{lemma}
\label{lemmapbetaplusasymp}
Suppose that $\rho^* \neq 1/2$. It holds that $P_\alpha(\beta_+(\alpha)) \to +\infty$ as $|\alpha| \to +\infty$.
\end{lemma}
\begin{proof}
We have 
\begin{align*}
    P_\alpha(\beta_+(\alpha)) = \frac{2}{27} \left[\,(\rho^*\, (3-2\rho^*))^{\frac{3}{2}} \,|\alpha|^3 + (10 {\rho^*}^3 - 9 {\rho^*}^2) \,\alpha^3  \right] + o(\alpha^3)
\end{align*}
as $|\alpha| \to + \infty$. As $\alpha \to +\infty$, we have $P_\alpha(\beta_+(\alpha)) \to + \infty$ if and only if $({\rho^*}\, (3-2{\rho^*}))^{\frac{3}{2}} + (10{\rho^*}^3 - 9 {\rho^*}^2) > 0$. The latter inequality is equivalent to $10 -\frac{9}{{\rho^*}} + \left(\frac{3}{{\rho^*}} - 2\right)^{\frac{3}{2}} > 0$. Set $f(x) = 10 - 9x + (3x - 2)^{\frac{3}{2}}$, for $x > 1$. Then $f(x) \to + \infty $ as $x \to +\infty$ and $f^\prime(x) = 9 \left(-1 + \frac{1}{2} (3x-2)^\frac{1}{2}\right) = 0$ if and only if $x=2$. Further, we have $f(2) = 0$. Therefore, the minimal value of $f(x)$ is $0$, and it is strictly attained at $x=2$. Provided that ${\rho^*} \neq \frac{1}{2}$, it follows that $P_\alpha(\beta_+(\alpha)) \to + \infty$ when $\alpha \to +\infty$. On the other hand, when $\alpha \to -\infty$, we have $P_\alpha(\beta_+(\alpha)) = \frac{2 |\alpha|^3}{27} \big[\,({\rho^*}\, (3-2{\rho^*}))^{\frac{3}{2}} + 9 {\rho^*}^2 - 10 {\rho^*}^3 \big] + o(\alpha^3)$. Further, $({\rho^*}\, (3-2{\rho^*}))^{\frac{3}{2}} + 9 {\rho^*}^2 - 10 {\rho^*}^3 > 0$ if and only $\frac{9}{{\rho^*}} -10 + \left(\frac{3}{{\rho^*}} - 2\right)^{\frac{3}{2}} > 0$. Setting $g(x) = 9x - 10 + (3x - 2)^{\frac{3}{2}}$, for $x \geq 1$, we have that $g^\prime(x) = 9 + \frac{9}{2}\sqrt{3x-2} > 0$, and $g(1) = 0$. Therefore, $P_\alpha(\beta_+(\alpha)) \to + \infty$ when $\alpha \to -\infty$
\end{proof}
\begin{lemma}
\label{lemmarangealpha}
Suppose that $\rho^* \neq \frac{1}{2}$. Then the following statements are true.
\begin{enumerate}[(a)]
	\item Suppose that ${\rho^*} < \frac{1}{2}$. Then, there exist $\alpha_1, \alpha_2 \in \mathbb{R}$ such that $0 < \alpha_1 < 1 < \alpha_2$ and, for any $\alpha \geq 0$, $\beta_+(\alpha) < \rho^*$ if and only if $\alpha \in (\alpha_1, \alpha_2)$. Further, $\beta_+(\alpha_1) = \beta_+(\alpha_2) = \rho^*$.
	\item Suppose that $\rho^* > \frac{1}{2}$. Then, there exists $\overline{\alpha} \in (0,1)$ such that for any $\alpha \geq 0$, $0 < \beta_+(\alpha) < \rho^*$ if and only if $\alpha \in (\overline{\alpha}, 1)$. Further, $\beta_+(\overline{\alpha}) = {\rho^*}$ and $\beta_+(1) = 0$.
\end{enumerate}
\end{lemma}
\begin{proof}
Fix $\alpha \in \mathbb{R}$. Using the expression of $\beta_+(\alpha)$ given in Lemma~\ref{lemmamaximalrealroot}, we have that $\beta_+(\alpha) < \rho^*$ if and only if
\begin{align*}
\frac{1}{3}\left[{\rho^*} (1-2\alpha) + \sqrt{ \frac{{\rho^*}}{3-2{\rho^*}} \left[ ((3-2{\rho^*})\alpha + {\rho^*} - 3)^2 + 3{\rho^*}(1-{\rho^*}) \right] }\,\right] < {\rho^*}\,,
\end{align*}
which is equivalent to
\begin{align}
\label{EqnIntermediateInequality2}
\sqrt{ \frac{{\rho^*}}{3-2{\rho^*}} \left[ ((3-2{\rho^*})\alpha + {\rho^*} - 3)^2 + 3{\rho^*}(1-{\rho^*}) \right] } < 2 {\rho^*} (1+\alpha)\,.
\end{align}
Inequality~\eqref{EqnIntermediateInequality2} can only be true for $\alpha > -1$, which we assume from now on. Squaring both sides and after a few manipulations, we find that inequality~\eqref{EqnIntermediateInequality2} is equivalent to
\begin{align*}
Q(\alpha) \defn \alpha^2 (-1 + 2{\rho^*}) + 2 \alpha (1+{\rho^*}) + (-1+{\rho^*}) > 0\,.
\end{align*}
The polynomial $Q$ has two distinct real roots $\alpha_1, \alpha_2$, which are given by
\begin{align*}
\alpha_1 = \frac{{\rho^*} + 1 - \sqrt{{\rho^*} (5 -{\rho^*})}}{1-2{\rho^*}}\,, \qquad \alpha_2 = \frac{{\rho^*} + 1 + \sqrt{{\rho^*} (5 - {\rho^*})}}{1-2{\rho^*}}\,.
\end{align*}
If ${\rho^*} < \frac{1}{2}$, the dominant coefficient of $Q$ is negative, and $Q$ takes positive values between its two roots. Therefore, $\beta_+(\alpha) < {\rho^*}$ if and only if $\alpha \in (\alpha_1, \alpha_2) \cap (-1, +\infty)$. Further, it holds that $\alpha_1 > 0$, $\alpha_1 < 1$ and $\alpha_2 > 1$. Hence, $\beta_+(\alpha) < {\rho^*}$ if and only if $\alpha \in (\alpha_1, \alpha_2)$.\\ 
If ${\rho^*} > \frac{1}{2}$, the dominant coefficient of $Q$ is positive. Hence, $\beta_+(\alpha) < {\rho^*}$ if and only if $\alpha \in (\alpha_1, + \infty) \cup (-\infty, \alpha_2)$ and $\alpha > -1$. It holds that $\alpha_1 \in (0,1)$ and $\alpha_2 < -1$. Thus, setting $\overline{\alpha} = \alpha_1$, we have $\beta_+(\alpha) < {\rho^*}$ if and only if $\alpha > \overline{\alpha}$. On the other hand, a calculation yields that $\beta_+(\alpha) > 0$ if and only if ${\rho^*} (1-\alpha) (1+\alpha(2{\rho^*} - 1)) > 0$. Since ${\rho^*} > \frac{1}{2}$, it follows that ${\rho^*} (1+\alpha(2{\rho^*} - 1)) > 0$, and thus, $\alpha$ must be less than $1$. Hence, $\beta_+(\alpha) \in (0, {\rho^*})$ if and only if $\alpha \in (\overline{\alpha},1)$.
\end{proof}

\section*{Acknowledgments}
This work was partially supported by the National Science Foundation under grants IIS-1838179 and ECCS-2037304, Facebook Research, Adobe Research and Stanford SystemX Alliance. The authors thank Emmanuel Cand\`es, Edgar Dobriban, Michał Dereziński and Michael Mahoney for helpful discussions.

\end{document}